\documentclass[reqno,11pt]{amsart}
\usepackage{amssymb}
\usepackage{verbatim}

\newtheorem{theorem}{Theorem}[section]
\newtheorem{corollary}[theorem]{Corollary}
\newtheorem{lemma}[theorem]{Lemma}
\newtheorem{proposition}[theorem]{Proposition}

\theoremstyle{definition}
\newtheorem{remark}[theorem]{Remark}

\theoremstyle{definition}

\theoremstyle{definition}

\def\bR{\mathbb{R}}

\def\cC{\mathcal{C}}

\def\cM{\mathcal{M}}

\DeclareMathOperator{\osc}{osc}

\makeatletter
\def\dashint{\operatorname%
{\,\,\text{\bf--}\kern-.98em\DOTSI\intop\ilimits@\!\!}}
\makeatother

\newcommand{\mysection}[1]{\section{#1}
\setcounter{equation}{0}}

\begin{document}
\title[Non-local elliptic equations]{On $L_p$-estimates for a class of non-local elliptic equations}

\author[H. Dong]{Hongjie Dong}
\address[H. Dong]{Division of Applied Mathematics, Brown University,
182 George Street, Providence, RI 02912, USA}
\email{Hongjie\_Dong@brown.edu}
\thanks{H. Dong was partially supported by the NSF under agreements DMS-0800129 and DMS-1056737.}

\author[D. Kim]{Doyoon Kim}
\address[D. Kim]{Department of Applied Mathematics, Kyung Hee University, 1 Seocheon-dong, Giheung-gu, Yongin-si, Gyeonggi-do 446-701, Republic of Korea}
\email{doyoonkim@khu.ac.kr}
\thanks{D. Kim was supported by Basic Science Research Program through the National Research Foundation of Korea (NRF) funded by the Ministry of Education, Science and Technology (2011-0013960).}

\subjclass[2010]{45K05,35B65,60J75}

\keywords{non-local elliptic equations, Bessel potential spaces, L\'evy processes, the martingale problem.}

\begin{abstract}
We consider non-local elliptic operators with kernel $K(y)=a(y)/|y|^{d+\sigma}$,
where $0 < \sigma < 2$ is a constant and $a$ is a bounded measurable function.
By using a purely analytic method, we prove the continuity of the non-local operator $L$ from the Bessel potential space $H^\sigma_p$ to $L_p$, and the unique strong solvability of the corresponding non-local elliptic equations in $L_p$ spaces. As a byproduct, we also obtain interior $L_p$-estimates.
The novelty of our results is that the function $a$ is not necessarily to be homogeneous, regular, or symmetric. An application of our result is the uniqueness for the martingale problem associated to the operator $L$.
\end{abstract}

\maketitle

\mysection{Introduction}
Non-local equations such
as integro-differential equations for jump L\'evy processes have
attracted the attention of many mathematicians.
These equations arise from models in physics,
engineering, and finance that involve long-range interactions (see, for instance, \cite{CT04}).
An example is the following non-local elliptic equation associated with pure jump process (see, for instance, \cite{MP92}):
\begin{equation}
                                \label{elliptic}
Lu-\lambda u=f\quad \text{in}\,\,\bR^d,
\end{equation}
where
\begin{equation}
                                    \label{eq23.22.24}
L u = \int_{\bR^d} \left(u(x+y) - u(x)-y\cdot \nabla u(x)\chi^{(\sigma)}(y)\right)K(x,y)\, d y,
\end{equation}
$$
\chi^{(\sigma)}\equiv 0\quad\text{for}\,\,\sigma\in (0,1),
\quad \chi^{(1)}=1_{y\in B_1},\quad
\chi^{(\sigma)}\equiv 1\quad\text{for}\,\,\sigma\in (1,2).
$$
In the above, $\lambda$ is a nonnegative constant and $K(x,y)$ is a positive kernel which has the following lower and upper bounds:
\begin{equation}
								\label{eq1211}	
(2-\sigma)\frac{\nu}{|y|^{d+\sigma}}
\le K(x,y) \le (2-\sigma)\frac{\Lambda}{|y|^{d+\sigma}},
\end{equation}
where $0<\nu\le \Lambda<\infty$ are two constants.

As is well known, if $K(x,y)=c^{-1}|y|^{-d-\sigma}$ with $c=c(d,\sigma)>0$ and $\sigma\in (0,2)$,
we get the fractional Laplace operator $-(-\Delta)^{\sigma/2}$, which has the symbol $- |\xi|^{\sigma}$.
In this case, the classical theory for pseudo-differential operators shows that, for any $\lambda > 0$ and $f \in L_p(\bR^d)$, $1<p<\infty$, there exists a unique solution $u \in H_p^{\sigma}(\bR^d)$ to the equation \eqref{elliptic} satisfying
$$
\|u\|_{H_p^\sigma(\bR^d)}
\le N(d,\sigma,\lambda, p) \|f\|_{L_p(\bR^d)};
$$
see Section \ref{sec1.2} for the definition of the Bessel potential space $H_p^\sigma(\bR^d)$.
In general, if the symbol of the operator is sufficiently smooth and its derivatives satisfy appropriate decays,
the aforementioned $L_p$-solvability is classical following from the Fourier multiplier theorems (see, for instance \cite{St70,Ja01,Ho}). It should be pointed out the $L_p$-solvability is also available if the kernel $K(y)$ is of the form $a(y)/|y|^{d+\sigma}$, and $a(y)$ is homogeneous of order zero and sufficiently smooth; see \cite{Ko84,MP92}.

In this paper, as a first step of our project, we extend this type of $L_p$-solvability to the equation \eqref{elliptic}\footnote{One can also consider the equation \eqref{elliptic} with $\chi^{(\sigma)} = 1_{y \in B_1}$ for all $\sigma \in (0,2)$. For a discussion about this case, see Remark \ref{remark1024}.} when the kernel $K$ is translation invariant with respect to $x$, i.e., $K(x,y)=K(y)$, merely measurable in $y$, and satisfies only the ellipticity condition \eqref{eq1211}. Moreover, if $\sigma=1$ we make a natural cancellation assumption on $K$; see \eqref{eq21.47}.
Note that the operator $L$ has the symbol
$$
m(\xi) = \int_{\bR^d} \left(e^{iy\cdot\xi } - 1-iy\cdot \xi\chi^{(\sigma)}(y) \right) K(y) \, dy,
$$
which generally lacks sufficient differentiability to apply the classical multiplier theorems.

There has been considerable work concerning regularity issues of solutions to non-local equations, such as the Harnack inequality, H\"{o}lder estimates, and non-local versions of the Aleksandrov--Bakelman--Pucci (ABP) estimate.
Firstly appeared approaches were probabilistic; see, for example, \cite{Bas02, Bas05_1, Bas05_2}.
Recently, analytic and PDE techniques have been used to study non-local equations with symmetric kernels in \cite{CS09, Ka09}, and with non-symmetric kernels in \cite{Si06, KL10, Ba11, Ba11b}.\footnote{The kernel $K$ is said to be symmetric if $K(y)=K(-y)$. In this case, $Lu$ can be written as
$$
Lu(x) = \frac12  \int_{\bR^d} \left(u(x+y) + u(x-y) -2 u(x)\right) K(y) \, d y.
$$
}
%The ideas for obtaining local regularities, espeically in the analytic approaches, play an important role in our arguments for the $L_p$-estimate.}
See also \cite{GS_10} for another ABP type estimate for a certain class of fully nonlinear non-local elliptic equations.

On the other hand, to the best of our knowledge,
little is known in the literature about the $L_p$-estimates of non-local operators if $K$ is only measurable and non-symmetric.
Our approach in this paper is purely analytic and uses techniques only from PDE, and does not use any multipliers or probabilistic representations of solutions.
We obtain a fully equipped $L_p$-estimate which enables us to get the desired $L_p$-solvability of the equation \eqref{elliptic} in the space $H_p^\sigma(\bR^d)$, $\sigma \in (0,2)$; see Theorem \ref{mainth01}.
We note that, in the symmetric case, a related $L_p$-estimate can be deduced from the main result in a fairly recent paper \cite{Ba07},
where a probabilistic approach is used to study Fourier multipliers.
To be precise, thanks to the symmetry of $K(y)$, applying Theorem 1 in \cite{Ba07} to the symbol
$$
M(\xi)=\frac{\int_{\bR^d}(\cos(\xi\cdot y)-1)a^{-1}(y)V(dy)}{\int_{\bR^d}(\cos(\xi\cdot y)-1)V(dy)},\quad V(dy)=K(y)\,dy,
$$
gives
\begin{equation*}
								%\label{eq1101}
\|u\|_{\dot H_p^\sigma(\bR^d)}
\le N \|Lu\|_{L_p(\bR^d)};
\end{equation*}
see Section \ref{sec1.2} for the definition of the homogeneous space $\dot H_p^\sigma(\bR^d)$.
%- History of $L_p$ theory of non-local equations: non-divergence and divergence type.

Our proof of $L_p$-estimates for non-local operators is founded on so-called {\em mean oscillation estimates} along with the Hardy--Littlewood maximal function theorem and the Fefferman--Stein theorem.
This method was used by N.V. Krylov in \cite{Krylov_2005} to treat second-order elliptic and parabolic equations with ${\rm VMO}_x$ coefficients (see also \cite{Iw83,DM93} for earlier work), and further developed in a series of papers including \cite{Krylov_2007} and \cite{DK09_01} for second-order and higher-order equations with rough coefficients.
In this paper, we adapt this method to study non-local operators.
One feature of the method is that it does not require a representation formula of solutions via fundamental solutions, which makes it possible to deal with  non-local operators with inhomogeneous and merely measurable kernels.
The key step in establishing the mean oscillation estimates of solutions is based on the following $C^{\alpha}$-estimate for the non-local equation $Lu - \lambda u = f$:
$$
[u]_{C^{\alpha}(B_{1/2})} \le N \int_{\bR^d}\frac{|u|}{1+|x|^{d+\sigma}}\,dx + N\osc_{B_1}|f|,
$$
with a constant $N$ which is independent of the size of $\lambda \ge 0$; cf. Corollary \ref{cor2.3}. This estimate is non-local in the sense that the local H\"older norm of the solution $u$
depends on $u$ itself in the whole space. For the proof, we use some ideas from \cite{Ba11}. To proceed from this H\"older estimate to the mean oscillation estimate of $u$, we make a crucial observation that the first term on the right-hand side above can be bounded by the maximal function of $u$ at the origin. We then use this idea to further estimate the mean oscillation of the fractional derivative $(-\Delta)^{\sigma/2}u$.

We remark that during the preparation of this paper we learned that Mikulevicius and Pragarauskas established $L_p$-estimates for non-local parabolic equations in \cite{MP10}, where they considered both stochastic local and non-local equations using probabilistic methods.
The ellipticity condition in \cite{MP10} is slightly more general than ours replacing $\nu$ in \eqref{eq1211} by a sufficiently smooth, positive, and homogeneous of order zero function, which can be degenerate on the whole space except on an arbitrarily narrow cone with vertex at zero (also see \cite{MP92}). However, for the strong solvability, the authors of \cite{MP10} appealed to the continuity estimate of $L$ proved in \cite{Ba07} and \cite{MP92}, which requires either the symmetry of $K$ or the homogeneity and sufficient smoothness of $K$. A direct consequence of our main result is the strong solvability of the stochastic non-local equations under considerably relaxed conditions; see Remark \ref{rem11.19}.

We state the main result, Theorem \ref{mainth01}, and its applications in the next section after we introduce a few necessary notation. The proof of Theorem \ref{mainth01} will be given in Section \ref{L_p}
after we prove an $L_2$-estimate in Section \ref{L_2},
a H\"older estimate in Section \ref{Holder}, and finally mean oscillation estimates in Section \ref{moe}. Section \ref{sec6} is devoted to several interior local estimates, which are deduced from the global estimate in Theorem \ref{mainth01}.

\mysection{Main result}

\subsection{Function spaces and notation}       \label{sec1.2}
For $p\in (1,\infty)$ and $\sigma>0$, we use $H_p^{\sigma}(\bR^d)$ to denote the Bessel potential space
\begin{align*}
H_p^\sigma(\bR^d)% &= \{ u \in L_p(\bR^d):
%\cF^{-1}\big((1+|\xi|^2)^{\sigma/2} \tilde u\big) \in L_p(\bR^d) \}\\
= \{ u \in L_p(\bR^d): (1-\Delta)^{\sigma/2} u \in L_p(\bR^d) \},
\end{align*}
which is equipped with the norm
$$
\|u\|_{H_p^\sigma(\bR^d)}=\|(1-\Delta)^{\sigma/2} u\|_{L_p(\bR^d)}.
$$
The homogeneous space is denoted by\begin{align*}
\dot H_p^\sigma(\bR^d) %&= \{ u \in L_p(\bR^d): \cF^{-1}\big(|\xi|^\sigma \tilde u\big) \in L_p(\bR^d) \}\\
= \{ u \in S'(\bR^d): (-\Delta)^{\sigma/2} u \in L_p(\bR^d) \},
\end{align*}
where $S'(\bR^d)$ is the space of tempered distributions. We use the semi-norm
$$
\|u\|_{\dot H_p^\sigma(\bR^d)}=\|(-\Delta)^{\sigma/2} u\|_{L_p(\bR^d)}.
$$
Note that by the inequalities
$$
N_1 (1+|\xi|^\sigma) \le (1+|\xi|^2)^{\sigma/2} \le N_2 (1+|\xi|^\sigma),
$$
we have
$$
\|u\|_{H_p^\sigma(\bR^d)} \approx \|u\|_{L_p(\bR^d)} + \|u\|_{\dot H_p^\sigma(\bR^d)}.
$$
Throughout the paper we omit $\bR^d$ in $C_0^\infty(\bR^d)$, $L_p(\bR^d)$, or $H_p^\sigma(\bR^d)$ whenever the omission is clear from the context.
We write $N(d,\nu,...)$ in the estimates to express that the constant $N$ is determined only by the parameters $d, \nu, ...$.

\subsection{Main theorem}

In addition to the ellipticity condition \eqref{eq1211}, in the case $\sigma=1$ we assume
\begin{equation}
                                    \label{eq21.47}
\int_{\partial B_r}yK(y)\, dS_r(y)=
0,\quad \forall r\in (0,\infty),
\end{equation}
where $dS_r$ is the surface measure on $\partial B_r$.
We remark that \eqref{eq21.47} is needed even for the continuity of $L$ from $H^\sigma_2$ to $L_2$; cf. Lemma \ref{thm1.2}. In particular, \eqref{eq21.47} is always satisfied for any symmetric kernels. It is worth noting that due to \eqref{eq21.47} the indicator function $\chi^{(1)}$ can be replaced by $1_{B_r}$ for any $r>0$.

Here is the main result of this paper.

\begin{theorem}[$L_p$-solvability]
								\label{mainth01}
Let $1< p < \infty$, $\lambda \ge 0$, and $0< \sigma < 2$. Assume that $K=K(y)$ satisfies \eqref{eq1211} and, if $\sigma=1$, $K$ also satisfies the condition \eqref{eq21.47}. Then $L$ defined in \eqref{eq23.22.24} is a continuous operator from $H_p^\sigma$ to $L_p$.
For $u \in H_p^{\sigma}$ and $f \in L_p$ satisfying
\begin{equation}
								\label{eq1220}
Lu - \lambda u = f	\quad\text{in}\,\,\bR^d,
\end{equation}
we have
\begin{equation}
                                \label{eq24.09.53}
\|u\|_{\dot H_p^\sigma} + \sqrt{\lambda}\|u\|_{\dot{H}^{\sigma/2}_p}
+ \lambda \|u\|_{L_p}
\le N \|f\|_{L_p},
\end{equation}
where $N=N(d,\nu,\Lambda,\sigma,p)$.
Moreover, for any $\lambda > 0$ and $f \in L_p$, there exists a unique strong solution $u \in H_p^{\sigma}$ of \eqref{eq1220}.
\end{theorem}

%Since $C_0^\infty$ is dense in $H_p^\sigma$, one understands the left-hand side of \eqref{eq1220} as a limit of $(L-\lambda)u_n$ in $L_p$ for a sequence $\{u_n\}\subset C_0^\infty$ converging to $u$ in $H_p^\sigma$ if the limit is uniquely well-defined.

\begin{remark}
Upon using the embedding $C^0\subset H^\sigma_p$ for $p>d/\sigma$, Theorem \ref{mainth01} implies a new uniqueness result for the martingale problem associated with the L\'evy type operator $L$; see, for instance, \cite{Ko84}. For other results about the martingale problem for pure jump processes, we refer the reader to \cite{Ko84b,NT89,MP93,AK09} and the references therein.
\end{remark}

\begin{remark}
For the sake of brevity, in this paper we do not present the precise dependence of the constant $N$ in \eqref{eq24.09.53} on the regularity parameter $\sigma$.
Nevertheless, by keeping track of the constants we find that, if $\sigma \in [\sigma_0, 2)$, where $\sigma_0 \in (0,2)$,
in the symmetric case the constant $N$ in the estimate \eqref{eq24.09.53} depends on $\sigma_0$, not $\sigma$.
In the non-symmetric case, if $0 < \sigma_0 \le \sigma \le \sigma_1 < 1$
or $1 < \sigma_0 \le \sigma < 2$, then the constant $N$ depends on $\sigma_0$ (and $\sigma_1$), not $\sigma$. In particular, $N$ does not blow up as $\sigma$ approaches $2$. A similar fact is observed in the study of local regularities of non-local equations in \cite{CS09}.
\end{remark}

\begin{remark}
                                    \label{rem11.19}
One noteworthy result in Theorem \ref{mainth01} is the continuity of the operator $L$ from $H_p^\sigma$ to $L_p$.
One can see from the proofs below that for this continuity the lower bound in the ellipticity condition \eqref{eq1211} is not needed.
This implies that the operators in \cite{MP10} are continuous from $H_p^\sigma$ to $L_p$ under Assumption A \cite{MP10} and the cancellation condition \eqref{eq21.47} in the case $\sigma=1$.
On the other hand, in \cite{MP10} it is shown that weak solutions are strong solutions if the operators are continuous.
Therefore, the weak solutions obtained in \cite{MP10} are indeed strong solutions (under the additional cancellation condition \eqref{eq21.47} when $\sigma=1$).
\end{remark}

A natural question is whether the result in Theorem \ref{mainth01} can be extended to equations with translation-variant kernels of the form $K(x,y) = a(x,y) |y|^{-d-\sigma}$, under natural conditions on $K$, say $K$ satisfies the assumptions above and $a$ is uniformly continuous (or smooth) with respect to $x$. Recall that the classical $L_p$-theory for second-order equations with uniformly continuous coefficients is built upon the estimates for equations with constant coefficients by using a standard perturbation argument and a partition of unity technique. However, for the non-local operator \eqref{eq23.22.24}, such a perturbation method seems to be out of reach. We note that estimates of this type were obtained in \cite{MP92} by using the Calder\'{o}n--Zygmund approach when the function $a(x,y)$ is homogeneous in $y$ of order zero and (some higher order) derivatives of $a(x,y)$ in $y$ are uniformly continuous in $x$. The $L_p$-estimate in the translation-variant case remains to be a challenging problem if $a(x,y)$ is inhomogeneous and merely measurable with respect to $y$.

\begin{remark}							\label{remark1024}
In our main theorem (Theorem \ref{mainth01}), we consider the operator $L$ in \eqref{eq23.22.24} with three different $\chi^{(\sigma)}$ depending on the range of $\sigma$. In this remark, we discuss the solvability in the unified case $\chi^{(\sigma)} = 1_{y \in B_1}$ for all $\sigma \in (0,2)$, which is also of interest from the probabilistic point of view.
Upon setting
\begin{equation}
                                    \label{eq23.22.24b}
\tilde L u = \int_{\bR^d} \left(u(x+y) - u(x)-y\cdot \nabla u(x)1_{y\in B_1}\right)K(y)\, d y,
\end{equation}
we observe that
\begin{equation}
                                    \label{eq10.28}
\tilde Lu=Lu+b\cdot \nabla u,
\end{equation}
where
$$
b=-\int_{B_1}yK(y)\,dy\quad \text{if}\,\,\sigma\in (0,1),
\quad b=\int_{\bR^d\setminus B_1}yK(y)\,dy\quad \text{if}\,\,\sigma\in (1,2).
$$
Then the unique solvability in $H_p^\sigma$ of $\tilde{L} u - \lambda u = f$ follows from that of the equation
\begin{equation}
								\label{eq8.56}
Lu +b\cdot \nabla u- \lambda u = f	\quad\text{in}\,\,\bR^d,
\end{equation}
where $b=(b_1,\ldots,b_d)$ is a constant vector.
For the equation \eqref{eq8.56}, as in the proof of Theorem \ref{mainth01}, it suffices to prove the following estimate for $u \in C_0^{\infty}$ satisfying \eqref{eq8.56}:
\begin{equation}
                                \label{eq8.57}
\|u\|_{\dot H_p^\sigma} + \sqrt{\lambda}\|u\|_{\dot{H}^{\sigma/2}_p}
+ \lambda \|u\|_{L_p}+\|b\cdot \nabla u\|_{L_p}
\le N \|f\|_{L_p},
\end{equation}
where $N = N(d,\nu,\Lambda, \sigma, p)$.
This estimate is proved using the results in \cite{MP10} combined with the continuity of the operator $L$ from $H_p^\sigma$ to $L_p$ proved in Theorem \ref{mainth01}.
For the reader's convenience, we present a proof at the end of Section \ref{L_p}. We note that, because of \eqref{eq10.28}, in general $\tilde L$ defined in \eqref{eq23.22.24b} is not a continuous operator from $H^\sigma_p$ to $L_p$ when $\sigma\in (0,1)$.
\end{remark}

\mysection{$L_2$-estimate}							\label{L_2}

To investigate the $L_p$-solvability of the equation $\eqref{eq1220}$, we first study an $L_2$-estimate. Recall that
$$
-(-\Delta)^{\sigma/2}u(x) = \frac 1 c \text{P.V.} \int_{\bR^d}\left(u(x+y)-u(x)\right)\frac{dy}{|y|^{d+\sigma}},
$$
where
\begin{equation}
								\label{eq1222}
c= c(d,\sigma)
= \frac{\pi^{d/2}2^{2-\sigma}}{\sigma(2-\sigma)}
\frac{\Gamma\left(2-\frac{\sigma}{2}\right)}{\Gamma
\left(\frac{d+\sigma}{2}\right)}.
\end{equation}
Here $\Gamma$ is the Gamma function.
Throughout the paper we always assume that $K=K(y)$ satisfies \eqref{eq1211} and, if $\sigma=1$, $K$ also satisfies the condition \eqref{eq21.47}.

\begin{lemma}
                        \label{thm1.2}
The operator $L$ defined in \eqref{eq23.22.24} is continuous from $H_2^\sigma$ to $L_2$. Let $\lambda\ge 0$ be a constant and $u \in H^\sigma_2$ satisfy
$$
Lu - \lambda u = f\quad \text{in}\,\,\bR^d,
$$
where $f \in L_2(\bR^d)$.
Then we have
\begin{equation}
                    \label{eq24.08.55}
\|u\|_{\dot H_2^\sigma} + \sqrt{\lambda}\|u\|_{\dot{H}^{\sigma/2}_2}
+ \lambda \|u\|_{L_2}
\le N(d,\nu) \|f\|_{L_2}.
\end{equation}
\end{lemma}

\begin{proof}
%For the sake of convenience, in the proof below we omit $\bR^d$ in the integral symbol $\int_{\bR^d}$.
We first consider the case $u\in C_0^\infty$.
By taking the Fourier transform of \eqref{eq23.22.24},
%$$
%Lu =\int_{\bR^d} \left(	u(x+y)-u(x)-y\cdot \nabla u(x)\chi^{(\sigma)}(y) \right) K(y) \, dy,
%$$
we have
$$
\widehat{Lu}(\xi) = \hat{u}(\xi)\int_{\bR^d} \left( e^{i\xi \cdot y} - 1 -iy\cdot\xi \chi^{(\sigma)}(y)\right) K(y) \, dy.
$$
Then
\begin{align*}
\int_{\bR^d} |Lu|^2 \, dx
&= \int_{\bR^d} |\widehat{Lu}(\xi)|^2 \, d \xi\\
&= \int_{\bR^{d}}  |\hat u(\xi)|^2 \Big|\int_{\bR^d}\left( e^{i\xi \cdot y} - 1 -iy\cdot\xi \chi^{(\sigma)}(y)\right) K(y) \, dy\Big|^2\, d\xi\\
&\ge \int_{\bR^{d}}  |\hat u(\xi)|^2 \Big|\Re \int_{\bR^d}\left( e^{i\xi \cdot y} - 1 -iy\cdot\xi \chi^{(\sigma)}(y)\right) K(y) \, dy\Big|^2\, d\xi\\
&= \int_{\bR^{d}}  |\hat u(\xi)|^2 \Big( \int_{\bR^d}\left(1-\cos(\xi \cdot y) \right) K(y) \, dy\Big)^2\, d\xi\\
&\ge (2-\sigma)^2 \nu^2 \int_{\bR^{d}}  |\hat u(\xi)|^2 \Big( \int_{\bR^d}\left(1-\cos(\xi \cdot y) \right) |y|^{-d-\sigma} \, dy\Big)^2\, d\xi\\
&= \nu^2c^2(2-\sigma)^2\int_{\bR^d} |(-\Delta)^{\sigma/2}u|^2\,dx,
\end{align*}
where $c$ is from \eqref{eq1222}.
Here we used the lower bound in \eqref{eq1211} and the fact that $1-\cos(\xi \cdot y)$ is non-negative.
Note that, for $\sigma \in (0,2)$, there exists $N=N(d)$ such that
$$
c(2-\sigma) = \pi^{d/2} \frac{2^{2-\sigma}}{\sigma}\frac{\Gamma(2-\frac{\sigma}{2})}{\Gamma(\frac{d+\sigma}2)}\ge N(d).
$$
Hence it follows that
\begin{equation}
								\label{eq0101}
\int |Lu(x)|^2 \, dx \ge N(d,\nu)\|u\|_{\dot{H}_2^{\sigma}}^2.
\end{equation}
%Thanks to the upper bound of $K$, we also have
%\begin{equation}
%								\label{eq0101z}
%\int |Lu(x)|^2 \, dx \le N(d,\Lambda)\|u\|_{\dot{H}_2^{\sigma}}^2,
%\end{equation}
%which implies that $L$ is a continuous operator from $H_2^\sigma$ to $L_2$.
Similarly,
\begin{align}
-\int_{\bR^d} uLu \, dx
&= -\int_{\bR^d} \widehat{Lu}(\xi)\overline{\hat u(\xi)} \, d \xi\nonumber\\
&= -\int_{\bR^{d}}  |\hat u(\xi)|^2 \int_{\bR^d}\left( e^{i\xi \cdot y} - 1 -iy\cdot\xi \chi^{(\sigma)}(y)\right) K(y) \, dy\, d\xi\nonumber\\
&= -\int_{\bR^{d}}  |\hat u(\xi)|^2 \Big(\Re \int_{\bR^d}\left( e^{i\xi \cdot y} - 1 -iy\cdot\xi \chi^{(\sigma)}(y)\right) K(y) \, dy\Big)\, d\xi\nonumber\\
&= \int_{\bR^{d}}  |\hat u(\xi)|^2 \int_{\bR^d}\left(1-\cos(\xi \cdot y) \right) K(y) \, dy\, d\xi\nonumber\\
&\ge (2-\sigma) \nu \int_{\bR^{d}}  |\hat u(\xi)|^2 \int_{\bR^d}\left(1-\cos(\xi \cdot y) \right) |y|^{-d-\sigma} \, dy\, d\xi\nonumber\\
                                \label{eq22.23.02}
&= \nu c (2-\sigma)\int_{\bR^d} |(-\Delta)^{\sigma/4}u|^2\,dx.
\end{align}
From the equality
$$
\int|Lu - \lambda u|^2\, dx = \int |f|^2 \, dx,
$$
we finally obtain the estimate \eqref{eq24.08.55} for $u\in C_0^\infty$ by collecting \eqref{eq0101} and \eqref{eq22.23.02}.

For the general case, we need to show the continuity of $L$. The symbol of $L$ is given by
$$
m(\xi)=\int_{\bR^d}\left(e^{iy\cdot \xi}-1-i y\cdot \xi(1_{\sigma\in (1,2)}+1_{y\in B_1}1_{\sigma=1})\right)K(y)\,dy.
$$
Clearly, $m(0)=0$. In the sequel, we assume $\xi\neq 0$.
By using the upper bound of $K$ in \eqref{eq1211} and the change of variable $y\to y/|\xi|$, it is easily seen that for $\sigma\in (0,1)$ or $\sigma\in (1,2)$, we have $|m(\xi)|\le N(d,\sigma,\Lambda)|\xi|^\sigma$. If $\sigma=1$, from \eqref{eq21.47} we get
$$
m(\xi)=\int_{\bR^d}\left(e^{iy\cdot \xi}-1-i y\cdot \xi 1_{|\xi| y\in B_1}\right)K(y)\,dy,
$$
which gives $|m(\xi)|\le N|\xi|$ by using the same argument. Therefore, in any case we have
\begin{equation}
                                        \label{eq0101z}
\|Lu\|_{L_2}=\|\hat u(\xi) m(\xi)\|_{L_2}\le N\|\hat u(\xi) |\xi|^\sigma\|_{L_2}
\le N \|u\|_{\dot H_2^\sigma},
\end{equation}
which implies that $L$ is a continuous operator from $H_2^\sigma$ to $L_2$.
To prove the estimate \eqref{eq24.08.55} for general $u\in H^\sigma_p$, we use the fact that $C_0^\infty$ is dense in $H_2^\sigma$ and the continuity of the operator $L-\lambda$ from $H_2^\sigma$ to $L_2$. This completes the proof of the lemma.
\end{proof}

\begin{remark}
We note that the proofs of \eqref{eq0101} and \eqref{eq22.23.02} do not use the cancellation condition when $\sigma=1$. These inequalities can also be verified without using the Fourier transform. Indeed, \eqref{eq22.23.02} follows from the identity
$$
-2\int u Lu\,dx=\int\int\big(u(x+y)-u(x)\big)^2K(y)\,dy\,dx
$$
and the ellipticity condition \eqref{eq1211}. For \eqref{eq0101}, we decompose $K$ into its symmetric and skew-symmetric parts $K=K_e+K_o$, where
$$
K_e(y)=\frac 1 2(K(y)+K(-y)),\quad K_o(y)=\frac 1 2(K(y)-K(-y)).
$$
Clearly, $K_e$ satisfies \eqref{eq1211}.
Let $L_e$ and $L_o$ be the corresponding operators with kernels $K_e$ and $K_o$, respectively. It is easily seen that
$$\int L_e u L_o u\,dx=0.
$$
Therefore, we have
$$
\int |Lu|^2\,dx\ge \int |L_e u|^2\,dx:=I.
$$
Since $K_e$ is symmetric,
\begin{equation*}
I= \iiint
\left(u(x+y)-u(x)\right)\left(u(x+z)-u(x)\right)K_e(y)K_e(z)\, dy \, dz \, dx.
\end{equation*}
Using the change of variables $y\to -y$ and $x \to x+y$, we get
$$
I = \iiint
\left(u(x)-u(x+y)\right)\left(u(x+y+z)-u(x+y)\right)K_e(y)K_e(z)\, dy \, dz \, dx.
$$
Adding the above two expressions of $I$ gives
\begin{align*}
2I
&= \iiint
\left(u(x)-u(x+y)\right)\left(u(x+y+z)-u(x+y)-u(x+z)+u(x)\right)\\
&\quad \cdot K_e(y)K_e(z)\, dy \, dz\, dx .
\end{align*}
Now we use the change of variables $z \to -z$ and $x \to x + z$ to obtain another expression of $2I$, which is the same as above with $u(x+y+z)-u(x+z)$
in place of $u(x)-u(x+y)$.
By adding these two expressions of $2I$, we finally reach
\begin{align*}
&4\int |Lu(x)|^2 \, dx\\
&= \iiint
\left(u(x+y+z)-u(x+y)-u(x+z)+u(x)\right)^2 K_e(y)K_e(z)\, dy \, dz \, dx,
%\\
%&\ge \frac{\nu^2(2-\sigma)^2}4\iiint
%\left(u(x+y+z)-u(x+y)-u(x+z)+u(x)\right)^2\\
%&\qquad \cdot \frac{1}{|y|^{d+\sigma}}\frac{1}{|z|^{d+\sigma}}\, dy \, dz \, dx\\
%&=\nu^2c^2(2-\sigma)^2\iiint
%\left(u(x+y)-u(x)\right)\left(u(x+z)-u(x)\right)\\
%&\qquad \cdot\frac{1}{c|y|^{d+\sigma}}
%\frac{1}{c|z|^{d+\sigma}}\, dy \, dz \, dx\\
%&=\nu^2c^2(2-\sigma)^2\int |(-\Delta)^{\sigma/2}u|^2\,dx,
\end{align*}
which along with \eqref{eq1211} gives \eqref{eq0101}.
\end{remark}

For the solvability result, we present the following two lemmas, which are versions of those in \cite[Chap. 1]{Krylov:book:2008} for non-local operators.
For later references, the operator in these lemmas is a bit more general than that in Theorem \ref{mainth01} having a drift term $b \cdot \nabla u$.
The first lemma is a maximum principle.

\begin{lemma}[A maximum principle]
                                        \label{lem2.2}
Let $\lambda>0$ be a constant, $b=(b_1,\ldots,b_d)$ be a bounded measurable function in $\bR^d$, and $u$ be a smooth function in $\bR^d$ satisfying $u(x)\to 0$ as $x\to \infty$. Assume that $Lu+b\cdot \nabla u-\lambda u=0$ in $\bR^d$. Then $u\equiv 0$ in $\bR^d$.
\end{lemma}
\begin{proof}
We prove the lemma by contradiction. Suppose that $\sup_{\bR^d}u>0$. Since $u$ tends to $0$ as $x\to \infty$, we can find $x_0\in \bR^d$ such that $u(x_0)=\sup_{\bR^d}u$. Then from \eqref{eq23.22.24}, it is easily seen that $Lu(x_0)\le 0$. This together with $u(x_0)>0$ and $\nabla u(x_0)=0$ gives $Lu-\lambda u<0$ at $x_0$, which contradicts the assumption in the lemma. Therefore, we must have $\sup_{\bR^d}u\le 0$. Similarly, $\inf_{\bR^d}u\ge 0$. This completes the proof of the lemma.
\end{proof}

\begin{lemma}
                                        \label{lem2.3}
Let $\lambda>0$ be a constant and $b=(b_1,\ldots,b_d)$ be a constant vector in $\bR^d$. Then the set $(L+b\cdot \nabla-\lambda)C_0^\infty$ is dense in $L_p$ for any $p\in (1,\infty)$.
\end{lemma}
\begin{proof}
Assume the assertion is not true. Then by the Hahn-Banach theorem and Riesz's representation theorem, there is a nonzero function $g\in L_{p/(p-1)}$ such that
\begin{equation}
							\label{eq1025}
\int_{\bR^d}(Lu(x)+b\cdot \nabla u(x)-\lambda u(x))g(x)\,dx=0
\end{equation}
for any $u\in C_0^\infty$.
Let $L^*$ be the non-local operator $L$ with $K(y)$ replaced by $K(-y)$. Then
we see that, for each $y \in \bR^d$,
\begin{multline*}
L^* (u*g)(y)-b\cdot \nabla (u*g)(y)-\lambda u*g(y)
\\
= \int_{\bR^d}(Lv(x)+b\cdot \nabla v(x)-\lambda v(x))g(x)\,dx=0,
\end{multline*}
where $v(x)=u(y-x) \in C_0^\infty$ and the last equality is due to \eqref{eq1025} with $v$ in place of $u$.
Because $u\in C_0^\infty$ and $g\in L_{p/(p-1)}$, the function $u*g(y)$ is smooth and tends to zero as $y\to \infty$. By Lemma \ref{lem2.2} applied to the operator $L^* - b \cdot \nabla - \lambda$, we get that $u*g\equiv 0$ in $\bR^d$. Bearing in mind that $u\in C_0^\infty$ is arbitrary, we conclude $g\equiv 0$ in $\bR^d$, which contradicts our assumption that $g$ is a nonzero function. The lemma is proved.
\end{proof}

Now we are ready to prove the following solvability result.

\begin{proposition}[$L_2$-solvability]
                        \label{prop2.5}
For any $\lambda > 0$ and $f \in L_2$, there exists a unique strong solution $u \in H_2^{\sigma}$
to $Lu - \lambda u =f$ in $\bR^d$ satisfying \eqref{eq24.08.55}.
\end{proposition}
\begin{proof}
Due to Lemma \ref{lem2.3}, we can find a sequence $u_n\in C_0^\infty$ such that $Lu_n-\lambda u_n$ converges to $f$ in $L_2$. By Lemma \ref{thm1.2}, we have
\begin{equation}
                                        \label{eq24.09.16}
\|u_n\|_{\dot H_2^\sigma} + \sqrt{\lambda}\|u_n\|_{\dot{H}^{\sigma/2}_2}
+ \lambda \|u_n\|_{L_2}
\le N(d,\nu) \|Lu_n-\lambda u_n\|_{L_2}
\end{equation}
and
\begin{multline*}
\|u_n-u_m\|_{\dot H_2^\sigma} + \sqrt{\lambda}\|u_n-u_m\|_{\dot{H}^{\sigma/2}_2}
+ \lambda \|u_n-u_m\|_{L_2}\\
\le N(d,\nu) \|L(u_n-u_m)-\lambda (u_n-u_m)\|_{L_2}.
\end{multline*}
Therefore, $\{u_n\}$ is a Cauchy sequence in $H_2^\sigma$ and there is a limiting function $u\in H_2^\sigma$. By the continuity estimate \eqref{eq0101z} and \eqref{eq24.09.16}, $u$ is a strong solution to $Lu - \lambda u =f$ and satisfies \eqref{eq24.08.55}. Finally, the uniqueness follows from the estimate \eqref{eq24.08.55}. The proposition is proved.
\end{proof}

\begin{remark}
In the proof of Proposition \ref{prop2.5}, instead of relying on Lemmas \ref{lem2.2} and \ref{lem2.3}, one may also use the method of continuity and the solvability of $-(-\Delta)^{\sigma/2} - \lambda u = f$ in $H_2^2$.
The same remark applies to the proof of Theorem \ref{mainth01}.
\end{remark}

\mysection{H\"older estimate}						\label{Holder}

In this section we prove a H\"{o}lder estimate of solutions to the equation $Lu - \lambda u =f$. The novelty of the result here is that the constant in the estimate is independent of $\lambda \ge 0$.
Our proof is based on the arguments developed in \cite{Ba11}. In the case $\lambda=0$, similar H\"{o}lder estimates with very different proofs can be found in \cite{CS09} for symmetric kernels and very recently in \cite{KL10} for non-symmetric kernels. We note that more general nonlinear Pucci type operators are treated in \cite{CS09,KL10}.

\begin{theorem}[$C^\alpha$-estimate]
								\label{thm1201}
Let $\lambda \ge 0$, $0 < \sigma < 2$, $1/2\le r<R<1$, and $f \in L_{\infty}(B_1)$.
Let $u \in C^{2}_{\text{loc}}(B_1)\cap L_1(\bR^d,\omega)$ with $\omega(x) = 1/(1+|x|^{d+\sigma})$
such that
$$
Lu - \lambda u = f
$$		
in $B_R$.
Then for any $\alpha \in (0, \min\{1,\sigma\})$,
we have
\begin{multline*}
[u]_{C^{\alpha}(B_r)}\\
\le N \Big( (R-r)^{-\alpha}\sup_{B_R}|u| + (R-r)^{-d-\alpha}\|u\|_{L_1(\bR^d, \omega)} + (R-r)^{\sigma-\alpha}\osc_{B_R}f\Big),
\end{multline*}
where $N=N(d,\nu,\Lambda,\sigma, \alpha)$.
\end{theorem}

\begin{proof}
Denote $r_1=(R-r)/2$, and $\bar r=(R+r)/2$.
Set $w(x) = I_{B_R}(x)u(x)$.
For $x \in B_{\bar r}$, we have $\nabla u(x)=\nabla w(x)$ and thus
\begin{align*}
Lu(x)&=\int_{\bR^d} \left(u(x+z) - u(x)-z\cdot \nabla u(x)\chi^{(\sigma)}(z)\right) K(z) \, dz\\
&= \int_{\bR^d}\left(w(x+z) - w(x)-z\cdot \nabla w(x)\chi^{(\sigma)}(z)\right) K(z) \, dz \\
&\quad + \int_{\bR^d}(u(x+z)- w(x+z))K(z)\, dz\\
&=Lw(x) + \int_{|z|\ge r_1}(u(x+z)- w(x+z))K(z)\, dz.
\end{align*}
Hence in $B_{\bar r}$
$$
\lambda w(x) - Lw(x) = g(x)-f(x),
$$
where
$$
g(x) = \int_{|z|\ge r_1}(u(x+z)- w(x+z))K(z)\,dz.
$$
Note that
\begin{equation}
								\label{eq1204}
\|g\|_{L_{\infty}(B_R)}
\le N r_1^{-d-\sigma}\|u\|_{L_1(\bR^d,\omega)},
\end{equation}
where $N=N(d,\Lambda,\sigma)$.

For $x_0 \in B_r$, we set
$$
M(x,y):= w(x) - w(y) - \phi(x-y) - \Gamma(x),
$$
where $\phi(z) = C_1 |z|^\alpha$, $\alpha \in (0,\min\{1,\sigma\})$, and $\Gamma(x) = C_2|x-x_0|^2$.
We will find $C_1, C_2 \in (0,\infty)$ depending only on $d$, $\nu$, $\Lambda$, $\sigma$, $\|u\|_{L_\infty(B_R)}$, $\|u\|_{L_1(\bR^d,\omega)}$, $\osc_{B_R}f$, $r_1$, but independent of the choice of $x_0 \in B_r$,
such that
\begin{equation}
								\label{eq1205}
\sup_{x, y \in \bR^d}M(x,y) \le 0.	
\end{equation}
This proves the assertion in the theorem. More specifically, using the fact the $C_1$ and $C_2$ are independent of the choice of $x_0 \in B_r$, we obtain
$$
|u(x) - u(y)| \le C_1 |x-y|^{\alpha},
\quad x, y \in B_r,
$$
where $C_1$ will be taken below to be the right-hand side of the H\"{o}lder estimate in the theorem.

To prove \eqref{eq1205}, we first take
$$
C_2 := 8r_1^{-2}\|u\|_{L_\infty(B_R)}.
$$
Then, for $x \in \bR^d\setminus B_{r_1/2}(x_0)$,
$$
w(x)-w(y) \le 2 \|u\|_{L_\infty(B_R)} \le C_2|x-x_0|^2.
$$
This shows that
\begin{equation}
								\label{eq1201}
M(x,y) \le 0, \quad
x \in \bR^d\setminus B_{r_1/2}(x_0).
\end{equation}

To get a contradiction, let us assume that there exist $x, y \in \bR^d$
such that $M(x, y) > 0$. By \eqref{eq1201} we know that $x \in B_{r_1/2}(x_0)\subset B_{(\bar r+r)/2}$.
Moreover, if $M(x, y) > 0$, then
\begin{equation}
								\label{eq1202}
w(x) - w(y) > C_1|x - y|^{\alpha},
\quad
\text{i.e.,}
\quad
|x - y|^\alpha < \frac{2 \|u\|_{L_\infty(B_R)}}{C_1}.	
\end{equation}
If we take a sufficiently large $C_1$ so that $C_1 \ge 2^{1+\alpha}r_1^{-\alpha}\|u\|_{L_\infty(B_R)}$,
the above inequalities show that $y \in B_{\bar r}$.
Therefore, the assumption that $M(x,y)>0$ for some $x,y \in \bR^d$ (and the continuity of $u$ on $B_R$) enables us to assume that there exist $\bar x, \bar y \in B_{\bar r}$ satisfying $\sup_{x,y \in \bR^d}M(x, y) = M(\bar x, \bar y)>0$.

Note that at $\bar x, \bar y \in B_{\bar r}$ we have
\begin{align*}
g(\bar y)-f(\bar y) &= \lambda w(\bar y) - Lw(\bar y),\\
- g(\bar x)+f(\bar x) &= - \lambda w(\bar x) + Lw(\bar x).
\end{align*}
Thus, upon observing $w(\bar y) - w(\bar x) < 0$, it follows that
\begin{multline}
								\label{eq1206}
-2\|g\|_{L_\infty(B_R)}-\osc_{B_R}f \le \lambda \left(w(\bar y) - w(\bar x)\right)
+ Lw(\bar x)-Lw(\bar y)
\\
\le Lw(\bar x)-Lw(\bar y):= I.
\end{multline}
We decompose $K$ into a symmetric part $K_1$ and non-symmetric part $K_2$, where
$$
K_1(z)=\min\{K(z),K(-z)\},\quad K_2(z)=K(z)-K_1(z).
$$
Clearly, the kernel $K_1$ also satisfies \eqref{eq1211}, and $K_2\ge 0$ has the upper bound in \eqref{eq1211}. Let $L_1$ and $L_2$ be the elliptic operators with kernels $K_1$ and $K_2$, respectively.
Then $I$ in \eqref{eq1206} can be written as
$$
I = I_1 + I_2,
$$
where
\begin{equation}
							\label{eq102401}
I_1 :=L_1w(\bar x)-L_1w(\bar y),
\quad
I_2 :=L_2w(\bar x)-L_2w(\bar y).
\end{equation}

Thanks to the symmetry of $K_1$, we have
$$
I_1=\frac 1 2\int_{\bR^d} J(\bar x, \bar y, z) K_1(z)\,dz,
$$
where
$$
J(\bar x, \bar y, z) = w(\bar x+z) + w(\bar x-z) - 2 w(\bar x) - w(\bar y+z) - w(\bar y-z) + 2 w(\bar y).
$$
Since $M(x,y)$ attains its maximum at $\bar x, \bar y$, we have
\begin{align}
&w(\bar x + z) - w(\bar y + z) - \phi(\bar x - \bar y) - \Gamma(\bar x + z)\nonumber\\
                        \label{eq05.16.51}
&\quad \le
w(\bar x) - w(\bar y) - \phi(\bar x- \bar y) - \Gamma(\bar x),
\end{align}
and
\begin{align*}
&w(\bar x - z) - w(\bar y - z) - \phi(\bar x - \bar y) - \Gamma(\bar x - z)\\
&\quad \le
w(\bar x) - w(\bar y) - \phi(\bar x- \bar y) - \Gamma(\bar x)
\end{align*}
for all $z \in \bR^d$.
These two inequalities lead us to
\begin{equation}
								\label{eqk1}
J(\bar x, \bar y, z) \le \Gamma (\bar x + z) + \Gamma(\bar x - z) - 2 \Gamma (\bar x),
\quad
z \in \bR^d.
\end{equation}
By again the assumption that $M(x,y)$ has the maximum at $\bar x, \bar y$, we have
\begin{align*}
&w(\bar x + z) - w(\bar y - z) - \phi(\bar x - \bar y + 2z) - \Gamma(\bar x + z)\\
&\quad\le
w(\bar x) - w(\bar y) - \phi(\bar x- \bar y) - \Gamma(\bar x),
\end{align*}
and
\begin{align*}
&w(\bar x - z) - w(\bar y + z) - \phi(\bar x - \bar y - 2z) - \Gamma(\bar x - z)\\
&\quad\le
w(\bar x) - w(\bar y) - \phi(\bar x- \bar y) - \Gamma(\bar x)
\end{align*}
for all $z \in \bR^d$.
Hence it follows that, for any $z \in \bR^d$,
\begin{multline}
								\label{eqk2}
J(\bar x, \bar y, z) \le
\phi(\bar x - \bar y + 2z) + \phi(\bar x - \bar y - 2z)- 2 \phi(\bar x - \bar y)  \\
+ \Gamma (\bar x + z) + \Gamma(\bar x - z) - 2 \Gamma (\bar x).			
\end{multline}

Set $a=\bar x - \bar y$.
Since $\bar x, \bar y$ satisfy \eqref{eq1202}, we have $|a| < r_1/2$.
Also set, for some $\eta_1, \eta_2 \in (0,1/2)$,
$$
\cC = \{ |z| <  \eta_1|a|: |z\cdot a| \ge (1-\eta_2) |a||z|\}.
$$
Then $\cC \subset B_{r_1}$ and
\begin{multline}
								\label{eq1207}
2I_1 = \int_{|z|\ge r_1} J(\bar x, \bar y, z) K_1(z) \, dz
+ \int_{B_{r_1} \setminus \cC}J(\bar x, \bar y, z) K_1(z) \, dz\\
+ \int_{\cC}J(\bar x, \bar y, z) K_1(z) \, dz
:= T_1 + T_2 + T_3.	
\end{multline}
Note that
$$
T_1 \le N(d,\Lambda,\sigma) r_1^{-\sigma}\|u\|_{L_\infty(B_R)}.
$$
By \eqref{eqk1} it follows
$$
T_2 \le \int_{B_{r_1} \setminus \cC} \left(\Gamma(\bar x + z) + \Gamma(\bar x - z) - 2 \Gamma(\bar x)\right) K_1(z) \, dz
\le Nr_1^{2-\sigma}C_2,
$$
where $N=N(d,\Lambda)$, but $N$ is independent of $\eta_1,\eta_2$ in the definition of $\cC$.
Now using \eqref{eqk2} we obtain
\begin{multline*}
T_3 \le \int_{\cC}
\left(\phi(\bar x - \bar y + 2z) + \phi(\bar x - \bar y - 2z) - 2 \phi(\bar x - \bar y)
\right)K_1(z)\,dz\\
+ \int_{\cC}\left(\Gamma (\bar x + z) + \Gamma(\bar x - z) - 2 \Gamma (\bar x)							 \right)K_1(z)\,dz
:= T_{3,1}+T_{3,2}.
\end{multline*}
The term $T_{3,2}$ is again bounded by $N r_1^{2-\sigma}C_2$, where $N = N(d,\Lambda)$.
Finally, by Lemma \ref{lem1201} below,
$$
T_{3,1} \le - N(d,\nu,\alpha) C_1 |a|^{\alpha-\sigma}.
$$
Thus, we get from \eqref{eq1207} and the choice of $C_2$ that
\begin{equation}
                                        \label{eq1207bb}
I_1\le N(d,\Lambda,\sigma)r_1^{-\sigma}\|u\|_{L_\infty(B_R)}-N(d,\nu,\alpha) C_1 |a|^{\alpha-\sigma}.
\end{equation}

Next we estimate $I_2 =L_2w(\bar x)-L_2w(\bar y)$ in \eqref{eq102401}.
We consider separately three cases: $\sigma<1$, $\sigma=1$, and $\sigma>1$.

{\em Case 1: $\sigma\in (0,1)$.}
In this case,
\begin{align}
                                            \label{eq5.17.14}
I_2&=\left(\int_{|z|\ge r_1}+\int_{B_{r_1}}\right)\left(w(\bar x+z)-w(\bar x)-w(\bar y+z)+w(\bar y)\right)K_2(z)\,dz\nonumber\\
&:=T_4+T_5.
\end{align}

Similar to $T_1$, we bound $T_4$ by $N(d,\Lambda,\sigma) r_1^{-\sigma}\|u\|_{L_\infty(B_R)}$.
Since $\sigma\in (0,1)$ and $|\bar x-x_0|<r_1/2$ by \eqref{eq1201}, from \eqref{eq05.16.51} we have
\begin{align*}
T_5&\le \int_{B_{r_1}}(\Gamma(\bar x+z)-\Gamma(\bar x))K_2(z)\,dz\\
&\le 2(2-\sigma)\Lambda C_2 \int_{B_{r_1}}(|z|^2+2|z||\bar x-x_0|)\frac{1}{|z|^{d+\sigma}}\,dz\\
&\le N(d,\Lambda,\sigma)r_1^{2-\sigma}C_2.
\end{align*}
Therefore, we get from \eqref{eq5.17.14} and the choice of $C_2$ that
\begin{equation}
                                        \label{eq5.17.54}
I_2\le N(d,\Lambda,\sigma)r_1^{-\sigma}\|u\|_{L_\infty(B_R)}.
\end{equation}

Combining \eqref{eq1206}, \eqref{eq1207bb}, \eqref{eq5.17.54}, and \eqref{eq1204}
we finally have
\begin{multline*}
0 \le N(d,\Lambda,\sigma)\left( \osc_{B_R}f + r_1^{-\sigma}\|u\|_{L_\infty(B_R)} + r_1^{-d-\sigma}\|u\|_{L_1(\bR^d,\omega)}\right) \\
- N(d,\nu,\alpha) C_1 |a|^{\alpha-\sigma}:=J.
\end{multline*}
Choose $C_1$ so that $C_1 \ge 2^{1+\alpha}r_1^{-\alpha} \|u\|_{L_{\infty}(B_R)}$ as well as
\begin{multline*}
C_1 \ge N(d,\Lambda,\sigma)r_1^{\sigma-\alpha}\Big(\osc_{B_R}f + r_1^{-\sigma}\|u\|_{L_\infty(B_R)} \\
+ r_1^{-d-\sigma}\|u\|_{L_1(\bR^d,\omega)}\Big)/ N(d,\nu,\alpha).
\end{multline*}
Then, for $\alpha \in (0,\min\{1,\sigma\})$, by \eqref{eq1202} $|a|^{\alpha-\sigma}r_1^{\sigma-\alpha} > 1$
and
\begin{multline*}
J \le N(d,\Lambda,\sigma)\Big(\osc_{B_R}f + r_1^{-\sigma}\|u\|_{L_\infty(B_R)} \\
+ r_1^{-d-\sigma}\|u\|_{L_1(\bR^d,\omega)}\Big)
(1-|a|^{\alpha-\sigma}r_1^{\sigma-\alpha}) < 0.
\end{multline*}
This contradicts the fact that $J\ge 0$.

{\em Case 2: $\sigma=1$.} Note that, because $K_1$ is symmetric, both $K_1$ and $K_2$ satisfy \eqref{eq21.47}. Therefore, $1_{B_1}$ can be replaced by $1_{B_{r_1}}$ in the definition of $L_2$, and we have $I_2=T_4+T_5$, where
\begin{align*}
T_4&=\int_{|z|\ge r_1}\big(w(\bar x+z)-w(\bar x)-w(\bar y+z)+w(\bar y)\big)K_2(z)\,dz,\\
T_5&=\int_{B_{r_1}}\big(w(\bar x+z)-w(\bar x)-w(\bar y+z)+w(\bar y)\\
&\qquad\quad -z\cdot(\nabla w(\bar x)-\nabla w(\bar y))\big)K_2(z)\,dz.
\end{align*}
Then we bound $T_4$ as in Case 1.

Since $M(x,y)$ attains its maximum at the interior point $(\bar x,\bar y)$, we easily get
\begin{equation}
								\label{equ01}
\nabla w(\bar x)=\nabla \phi(\bar x-\bar y)+\nabla \Gamma(\bar x),\quad
\nabla w(\bar y)=\nabla \phi(\bar x-\bar y).
\end{equation}
For $T_5$,  using \eqref{eq05.16.51} and \eqref{equ01}, we have
\begin{align*}
T_5&\le \int_{B_{r_1}}\big(\Gamma(\bar x+z)-\Gamma(\bar x)-z\cdot(\nabla w(\bar x)-\nabla w(\bar y))1_{B_1}\big)K_2(z)\,dz\\
&= \int_{B_{r_1}}\big(\Gamma(\bar x+z)-\Gamma(\bar x)-z\cdot\nabla \Gamma(\bar x)\big)K_2(z)\,dz\\
&= \int_{B_{r_1}}C_2|z|^2 K_2(z)\,dz\\
&\le N(d,\Lambda)r_1^{2-\sigma}C_2.
\end{align*}
Then we argue as in Case 1 to get the contradiction.

{\em Case 3: $\sigma\in (1,2)$.} Now $I_2=T_4+T_5$, where
\begin{align*}
T_4&=\int_{|z|\ge r_1}\big(w(\bar x+z)-w(\bar x)-w(\bar y+z)+w(\bar y)\\
&\qquad\quad -z\cdot(\nabla w(\bar x)-\nabla w(\bar y))\big)K_2(z)\,dz,\\
T_5&=\int_{B_{r_1}}\big(w(\bar x+z)-w(\bar x)-w(\bar y+z)+w(\bar y)\\
&\qquad\quad -z\cdot(\nabla w(\bar x)-\nabla w(\bar y))\big)K_2(z)\,dz.
\end{align*}
Because $\sigma\in (1,2)$, $|\bar x-x_0|<r_1/2$, and $C_2 = 8r_1^{-2}\|u\|_{L_\infty(B_R)}$, by \eqref{equ01} we have
\begin{align*}
T_4&\le \int_{|z|\ge r_1}\big(4\|u\|_{L_\infty(B_R)}+|z||\nabla \Gamma(\bar x)|\big)K_2(z)\,dz\\
&\le N(d,\Lambda,\sigma)r_1^{-\sigma}\|u\|_{L_\infty(B_R)}.
\end{align*}
It follows from \eqref{eq05.16.51} and \eqref{equ01} that
\begin{align*}
T_5&\le \int_{B_{r_1}}(\Gamma(\bar x+z)-\Gamma(\bar x)-z\cdot \nabla \Gamma(\bar x))K_2(z)\,dz\\
&= \int_{B_{r_1}}C_2|z|^2 K_2(z)\,dz\\
&\le N(d,\Lambda)r_1^{2-\sigma}C_2.
\end{align*}
So we again argue as in Case 1 to arrive at the contradiction.

Therefore, we conclude that \eqref{eq1205} holds true in all three cases.
The theorem is proved.
\end{proof}

Recall that $a = \bar x - \bar y$ and
$$	
T_{3,1}=\int_{\cC}
\left(\phi(a + 2z) + \phi(a - 2z) - 2 \phi(a)
\right)K_1(z)\,dz,
$$	
where
$$
\cC = \{ |z| <  \eta_1|a|: |z\cdot a| \ge (1-\eta_2) |a||z|\}.
$$

\begin{lemma}
								\label{lem1201}
There exist $\eta_1, \eta_2 \in (0,1/2)$, depending only on $\alpha$,
such that
\begin{equation}
								\label{eq1209}
T_{3,1} \le - N C_1 |a|^{\alpha-\sigma},	
\end{equation}
where $N=N(d,\nu,\alpha)>0$.
\end{lemma}

\begin{proof}
The idea of the proof is to use the local concavity of the function $|x|^\alpha$ in the radial direction.
Set $\eta(t) = a + 2 t z$, where $a = \bar x - \bar y$.	
Then
$$
\varphi(t):=\phi(a+2tz) = \phi(\eta(t)).
$$	
Since $\phi(x)=C_1|x|^\alpha$, we have
\begin{align*}
\frac{\partial \phi}{\partial x_i}(x) &= C_1 \frac{\partial}{\partial x_i}\left(|x|^\alpha\right)
= C_1 \alpha x_i |x|^{\alpha-2},\\
\frac{\partial^2 \phi}{\partial x_i \partial x_j}(x)
&= C_1 \alpha(\alpha-2) x_i x_j |x|^{\alpha-4}
+ C_1 \alpha|x|^{\alpha-2} I_{i=j}.
\end{align*}
Hence
$$
\varphi'(t) = \sum_{i=1}^d \frac{\partial \phi}{\partial x_i}(\eta(t)) \frac{d \eta_i(t)}{dt} =
\sum_{i=1}^d \frac{\partial \phi}{\partial x_i}(\eta(t)) 2 z_i
$$
and
\begin{align*}
\varphi''(t) &= \sum_{i,j=1}^d \frac{\partial^2 \phi}{\partial x_i \partial x_j}(\eta(t)) 4 z_i z_j\\
&= 4 C_1 \alpha(\alpha-2)|\eta(t)|^{\alpha-4} |\eta(t) \cdot z|^2 + 4 C_1 \alpha|\eta(t)|^{\alpha-2}|z|^2\\
&= 4 C_1 \alpha |a + 2 t z|^{\alpha-4}
\left[ (\alpha-2) |(a + 2 t z)\cdot z|^2 + |a+ 2 t z|^2 |z|^2 \right].
\end{align*}
%$$
%= 4 L_1 \alpha(\alpha-2)|a + 2 t z|^{\alpha-4} |(a + 2 t z)\cdot z|^2 + 4 L_1 \alpha|a + 2 t z|^{\alpha-2}|z|^2
%$$
Observe that, on $\cC$,
$$
|a+2 t z|^2 \le (1+2 \eta_1)^2|a|^2,
$$
$$
|(a+2 t z)\cdot z| = |a \cdot z + 2 t |z|^2|
\ge |a \cdot z| - 2 |z|^2
$$
$$
\ge (1-\eta_2)|a||z| - 2 |z|^2
\ge (1-2 \eta_1 - \eta_2)|z||a|
$$
for all $t \in [-1,1]$.
Thus upon noting $\alpha-2<0$ we get
\begin{equation}
								\label{eq1203}
\varphi''(t) \le 4 C_1 \alpha |a + 2 t z|^{\alpha-4}
\left[ (\alpha-2)(1-2\eta_1-\eta_2)^2 + (1+2\eta_1)^2 \right] |a|^2|z|^2.
\end{equation}
Since $(1-2\eta_1-\eta_2)^2 \to 1$ and $(1+2\eta_1)^2 \to 1$ as $\eta_1, \eta_2 \searrow 0$,
there exist sufficiently small $\eta_1, \eta_2 \in (0,1/2)$, depending only on $\alpha \in (0,1)$, such that
$$
(\alpha-2)(1-2\eta_1-\eta_2)^2 + (1+2\eta_1)^2
\le (\alpha-1)/2.
$$
This together with \eqref{eq1203} implies that
$$
\varphi''(t) \le - 2 C_1 \alpha (1-\alpha)|a + 2 t z|^{\alpha-4}|a|^2|z|^2.
$$
From this and the fact that
$$
|a+2 t z|^{\alpha-4} \ge (1+2 \eta_1)^{\alpha-4}|a|^{\alpha-4} \ge 2^{\alpha-4}|a|^{\alpha-4},
$$
we arrive at
\begin{equation}
								\label{eq1210}
\varphi''(t) \le -2^{\alpha-3} C_1 \alpha (1-\alpha)|a|^{\alpha-2}|z|^2,\quad
t \in [-1,1],
\quad z \in \cC.	
\end{equation}
On the other hand, by the mean value theorem for difference quotients,
there exists $t_0 \in (-1,1)$ satisfying
$$
\varphi(1) + \varphi(-1) - 2\varphi(0)
= \varphi''(t_0).
$$
Using this equality and \eqref{eq1210},
%the following identities
%$$
%\varphi(1)=\varphi(0) + \varphi'(0) + \int_0^1(1-t)\varphi''(t)\,dt,
%$$
%$$
%\varphi(-1)=\varphi(0)-\varphi'(0)+\int_{-1}^0(1+t)\varphi''(t)\,dt.
%$$
we have
\begin{equation}
								\label{eq1208}
T_{3,1}%=\int_{\cC}\left(\phi(a + 2z) + \phi(a - 2z) - 2 \phi(a)\right)K(z)\,dz
%= \int_{\cC}\left(\varphi(1) + \varphi(-1) - 2\varphi(0)\right)K(z) \, dz
%= \int_{\cC}\varphi''(t_0)K(z)\,dz
%\\
\le - \int_{\cC} 2^{\alpha-3} C_1 \alpha (1-\alpha)|a|^{\alpha-2}|z|^2 K_1(z) \, dz.
\end{equation}
From the definition of $\cC$ it follows that
$$
\int_{\cC} |z|^2 K_1(z) \, dz \ge \nu (2-\sigma) \int_{\cC} |z|^{2-d-\sigma} \, dz
 = N(d, \nu, \eta_2) \eta_1^{2-\sigma}|a|^{2-\sigma}.
$$
Combining this with \eqref{eq1208} and recalling the fact that $\eta_1$, $\eta_2$ depend only on $\alpha$, we finally obtain the inequality \eqref{eq1209}.
\end{proof}

In the next section we will need a bound of the $C^\alpha$ norm of $u$ only in terms of $f$ and the weighted $L_1$ norm of $u$.
To this end, in the corollary below
we use an iteration argument to drop the term $\sup_{B_R}|u|$ on the right-hand side of the estimate in Theorem \ref{thm1201}.

\begin{corollary}
                                    \label{cor2.3}
Let $\lambda \ge 0$, $0 < \sigma < 2$, and $f \in L_{\infty}(B_1)$.
Let $u \in C^{2}_{\text{loc}}(B_1)\cap L_1(\bR^d,\omega)$ with $\omega(x) = 1/(1+|x|^{d+\sigma})$
such that
$$
Lu - \lambda u = f
$$		
in $B_1$. Then for any $\alpha \in (0, \min\{1,\sigma\})$, we have
\begin{equation}
                                        \label{eq13.30}
[u]_{C^{\alpha}(B_{1/2})} \le N \|u\|_{L_1(\bR^d, \omega)} + N\osc_{B_1}f,
\end{equation}
where $N=N(d,\nu,\Lambda,\sigma, \alpha)$.
\end{corollary}
\begin{proof}
Set
$$
r_n = 1 -2^{-n-1}, \quad B_{(n)}=B_{r_n}, \quad n = 0,1, 2, \cdots.
$$
Theorem \ref{thm1201} gives, for $n=0,1,2,\cdots$,
\begin{equation}
                                    \label{eq12.36}
[u]_{C^{\alpha}(B_{(n)})} \le N_1 \Big( 2^{2n}\sup_{B_{(n+1)}}|u| +
2^{(d+\alpha)n}\|u\|_{L_1(\bR^d, \omega)} + \osc_{B_{(n+1)}}f\Big),
\end{equation}
where $N_1=N_1(d,\nu,\Lambda,\sigma, \alpha)$ is a constant independent of $n$. To estimate the first term on the right-hand side of \eqref{eq12.36}, by the well-known interpolation inequality, we have
\begin{equation}
                                    \label{eq13.15}
\sup_{B_{(n+1)}}|u|\le \varepsilon [u]_{C^{\alpha}(B_{(n+1)})}+N\varepsilon^{-d/\alpha} \|u\|_{L_1(B_{(n+1)})},\quad \forall \varepsilon\in (0,1).
\end{equation}
Upon taking $\varepsilon=(N_1 2^{2n+3d/\alpha})^{-1}$ and combining \eqref{eq12.36} and \eqref{eq13.15}, we get
\begin{align}
[u]_{C^{\alpha}(B_{(n)})} &\le 2^{-3d/\alpha}[u]_{C^{\alpha}(B_{(n+1)})}
+N2^{2nd/\alpha}\|u\|_{L_1(B_{1})}\nonumber \\
                                    \label{eq12.37}
&\quad +N2^{(d+\alpha)n}\|u\|_{L_1(\bR^d, \omega)} + N\osc_{B_{1}}f.
\end{align}
We multiply both sides of \eqref{eq12.37} by $2^{-3dn/\alpha}$ and sum over $n$ to obtain
\begin{align*}
&\sum_{n=0}^\infty 2^{-3dn/\alpha}[u]_{C^{\alpha}(B_{(n)})}\\
&\,\,\le \sum_{n=0}^\infty 2^{-3d(n+1)/\alpha}[u]_{C^{\alpha}(B_{(n+1)})}
+N\sum_{n=0}^\infty 2^{-dn/\alpha}\|u\|_{L_1(B_{1})}\\
&\quad +N\sum_{n=0}^\infty 2^{-3dn/\alpha+(d+\alpha)n}\|u\|_{L_1(\bR^d, \omega)} + N\sum_{n=0}^\infty 2^{-3dn/\alpha}\osc_{B_{1}}f,
\end{align*}
which immediately yields \eqref{eq13.30}. The corollary is proved.
\end{proof}

\mysection{Mean oscillation estimates}						\label{moe}

This section is devoted to several mean oscillation estimates for $u$ and its fractional derivative $(-\Delta)^{\sigma/2}u$ by using the $L_2$ estimate in Section \ref{L_2} and the H\"older estimate established in Section \ref{Holder}.

We recall the maximal function theorem and the Fefferman--Stein theorem.
Let the maximal and sharp functions of $g$ defined on $\bR^d$ be given by
\begin{align*}
\cM g (x) &= \sup_{r>0} \dashint_{B_r(x)} |g(y)| \, dy,\\
g^{\#}(x) &= \sup_{r>0} \dashint_{B_r(x)} |g(y) -
(g)_{B_r(x)}| \, dy.
\end{align*}
Then
\begin{equation}
								\label{eq1221}
\| g \|_{L_p} \le N \| g^{\#} \|_{L_p},
\quad
\| \cM g \|_{L_p} \le N \| g\|_{L_p},
\end{equation}
if $g \in L_p$, where $1 < p < \infty$ and $N = N(d,p)$.
As is well known, the first inequality above is due to the Fefferman--Stein theorem on sharp functions
and the second one to the Hardy--Littlewood maximal function theorem (this inequality also holds trivially when $p = \infty$).
Throughout the paper we denote
\begin{equation*}
(f)_{\Omega} = \frac{1}{|\Omega|} \int_{\Omega} f(x) \, dx
= \dashint_{\Omega} f(x) \, dx,
\end{equation*}
where $|\Omega|$ is the
$d$-dimensional Lebesgue measure of $\Omega$.

\begin{lemma}
                            \label{lem3.1}
Let $\lambda \ge 0$, $0<\sigma<2$, and $f \in C^\infty_{\text{loc}}\cap L_{\infty}(\bR^d)$ satisfying $f=0$ in $B_2$.
Let $u \in H_2^\sigma \cap C^{\infty}_{b}(\bR^d)$ satisfy
\begin{equation}
                            \label{eq13.43}
Lu - \lambda u = f\quad\text{in}\,\,\bR^d.
\end{equation}
Then for all $\alpha \in (0, \min\{1,\sigma\})$,
\begin{align}
                    \label{eq14.05}
[u]_{C^{\alpha}(B_{1/2})} &\le N\sum_{k=0}^{\infty} 2^{-k\sigma} (|u|)_{B_{2^k}},\\
                    \label{eq14.08}
[(-\Delta)^{\sigma/2}u]_{C^{\alpha}(B_{1/2})} &\le N
\left(\sum_{k=0}^{\infty} 2^{-k\sigma} (|(-\Delta)^{\sigma/2}u|)_{B_{2^k}} + \cM f(0)\right),
\end{align}
where $N=N(d,\nu,\Lambda,\sigma, \alpha)$.
\end{lemma}

Note that the right-hand side of \eqref{eq14.05} and the first term on the right-hand side of \eqref{eq14.08} are bounded by $\cM u(0)$ and $\cM((-\Delta)^{\sigma/2}u)(0)$, respectively. Therefore, Lemma \ref{lem3.1} implies that the local H\"older norms of $u$ and its fractional derivative $(-\Delta)^{\sigma/2}u$ can be controlled by the maximal functions of $u$, $(-\Delta)^{\sigma/2}u$, and $f$. This enables us to adapt the approach in \cite{Krylov_2005}.

\begin{proof}[Proof of Lemma \ref{lem3.1}]
First note that we have
$u,(-\Delta)^{\sigma/2}u\in C^2_{\text{loc}}(B_1)\cap L_1(\bR^d,\omega)$ with $\omega(x) = 1/(1+|x|^{d+\sigma})$.
Since $f=0$ in $B_2$, by Corollary \ref{cor2.3},
\begin{equation}
                                    \label{eq14.11}
[u]_{C^{\alpha}(B_{1/2})}
\le N \|u\|_{L_1(\bR^d, \omega)}.
\end{equation}
Set
$$
B_{(0)} = B_1,
\quad
B_{(k)} = B_{2^k}\setminus B_{2^{k-1}}, \quad k \ge 1.
$$
Note that
\begin{align*}
\|u\|_{L_1(\bR^d, \omega)}
&= \int_{\bR^d} |u(x)| \frac{1}{1+|x|^{d+\sigma}} \, dx\\
&= \sum_{k=0}^{\infty}\int_{B_{(k)}} |u(x)| \frac{1}{1+|x|^{d+\sigma}} \, dx\\
&\le N\sum_{k=0}^{\infty} 2^{-k\sigma} (|u|)_{B_{2^k}}.
%= N(d,\sigma_0)\cM u(0).
\end{align*}
This together with \eqref{eq14.11} gives \eqref{eq14.05}.

To prove \eqref{eq14.08}, we apply $(-\Delta)^{\sigma/2}$ to the both sides of \eqref{eq13.43} and obtain
$$
(L-\lambda)(-\Delta)^{\sigma/2}u = (-\Delta)^{\sigma/2}f.
$$
Again by Corollary \ref{cor2.3},
\begin{equation}
                                    \label{eq13.45}
[(-\Delta)^{\sigma/2}u]_{C^{\alpha}(B_{1/2})}
\le N \|(-\Delta)^{\sigma/2}u\|_{L_1(\bR^d, \omega)}
+ N \sup_{B_1}|(-\Delta)^{\sigma/2}f|.
\end{equation}
In exactly the same way above, we bound the first term on the right-hand side of \eqref{eq13.45} by
$$
N\sum_{k=0}^{\infty} 2^{-k\sigma} (|(-\Delta)^{\sigma/2}u|)_{B_{2^k}}.
$$
Next we estimate the second term on the right-hand side of \eqref{eq13.45}.
For $|x|<1$, we have
\begin{align}
\left|-(-\Delta)^{\sigma/2} f(x)\right|
&= \frac 1 c\left|\text{P.V.} \int_{\bR^d} \left(f(x+y)-f(x)\right) \frac 1 {|y|^{d+\sigma}}\, dy\right|\nonumber\\
                        \label{eq13.57}
&\le N \int_{|y|>1/2} |f(x+y)|\frac{1}{|y|^{d+\sigma}+1} \, dy,
\end{align}
where the inequality above is due to the fact that
$$
f(x) = 0 \quad \text{if} \quad |x| < 2,
\quad
f(x+y) = 0 \quad \text{if} \quad |x| < 1, \,\, |y| < 1/2.
$$
Similar to the estimate of $\|u\|_{L_1(\bR^d, \omega)}$ above, we bound the right-hand side of \eqref{eq13.57} by
$$
N\sum_{k=0}^{\infty} 2^{-k\sigma} (|f|)_{B_{2^k}}\le N(d,\sigma)\cM f(0).$$ The lemma is proved.
\end{proof}

By using a simple scaling argument, we obtain the following corollary.

\begin{corollary}
Let $\lambda \ge 0$, $0 < \sigma < 2$, $r>0$, $\kappa\ge 2$,
and $f \in C^\infty_{\text{loc}}\cap L_{\infty}(\bR^d)$ satisfying $f=0$ in $B_{2\kappa r}$.
Let $u \in H_2^\sigma \cap C^{\infty}_{b}(\bR^d)$ satisfy
\begin{equation*}
Lu - \lambda u = f \quad\text{in}\,\,\bR^d.
\end{equation*}
Then for all $\alpha \in (0, \min\{1,\sigma\})$,
$$
[u]_{C^{\alpha}(B_{\kappa r/2})} \le N (\kappa r)^{-\alpha}\sum_{k=0}^{\infty} 2^{-k\sigma} (|u|)_{B_{2^k \kappa r}},
$$
$$
[(-\Delta)^{\sigma/2}u]_{C^{\alpha}(B_{\kappa r/2})} \le N (\kappa r)^{-\alpha}
\left(\sum_{k=0}^{\infty} 2^{-k\sigma} (|(-\Delta)^{\sigma/2}u|)_{B_{2^k \kappa r}} + \cM f(0)\right),
$$
where $N=N(d,\nu,\Lambda,\sigma, \alpha)$.	
\end{corollary}

\begin{proof}
Let $R = \kappa r$, $w(x) = u(Rx)$, and $g(x) = R^\sigma f(Rx)$.
Set $L_1$ to be a non-local operator with the kernel $K_1(z) = R^{d+\sigma} K(Rz)$.
Then we see that $K_1$ satisfies \eqref{eq1211} and
$w \in L_1(\bR^d,\omega)$. %, where the $L_1(\bR^d,\omega)$-norm of $w$ depends on $\bR$.
Moreover,
$$
L_1 w - R^\sigma \lambda w = g \quad\text{in}\,\,\bR^d,
$$
where $g=0$ in $B_2$.
Applying Lemma \ref{lem3.1} to $w$, we obtain \eqref{eq14.05} and \eqref{eq14.08} with $w$ in place of $u$.
Turning $w$ back to $u$ gives the desired inequalities.
%In particular, we use
%$$
%(-\Delta)^{\sigma/2}w(x) = R^\sigma \left((-\Delta)^{\sigma/2}u\right) (Rx).
%$$
\end{proof}

Note that, for example,
$$
\left(|u-(u)_{B_r}|\right)_{B_r}
\le 2^\alpha r^\alpha [u]_{C^\alpha(B_{\kappa r/2})}
$$
for $\kappa \ge 2$.
This combined with the inequalities in the above corollary leads us to

\begin{corollary}
                            \label{cor3.2}
Let $\lambda \ge 0$, $0 < \sigma < 2$, $r>0$, $\kappa\ge 2$,
and $f \in C^\infty_{\text{loc}}\cap L_{\infty}(\bR^d)$ satisfying $f=0$ in $B_{2\kappa r}$.
Let $u \in H_2^\sigma \cap C^{\infty}_{b}(\bR^d)$ satisfy
\begin{equation*}
Lu - \lambda u = f \quad\text{in}\,\,\bR^d.
\end{equation*}
Then for all $\alpha \in (0, \min\{1,\sigma\})$,
\begin{align*}
\big(|u-(u)_{B_r}|\big)_{B_r} \le N\kappa^{-\alpha}\sum_{k=0}^{\infty} 2^{-k\sigma} (|u|)_{B_{2^k\kappa r}},
\end{align*}
\begin{multline*}
\big(|(-\Delta)^{\sigma/2}u-((-\Delta)^{\sigma/2}u)_{B_r}|\big)_{B_r}\\
\le N\kappa^{-\alpha}
\Big( \sum_{k=0}^{\infty} 2^{-k\sigma} \big(|(-\Delta)^{\sigma/2}u|\big)_{B_{2^k\kappa  r}} + \cM f(0)\Big),
\end{multline*}
where $N=N(d,\nu,\Lambda,\sigma, \alpha)$.
\end{corollary}

The proposition below is the main result of this section. It reads that the mean oscillations of $u$ and $(-\Delta)^{\sigma/2}u$ can be controlled by their maximal functions together with the maximal function of $f^2$.

\begin{proposition}[Mean oscillation estimate]
                            \label{prop3.3}
Let $\lambda > 0$, $0 <  \sigma < 2$, $r>0$, $\kappa\ge 2$,
and $f \in C^\infty_{\text{loc}}\cap L_{\infty}$.
Let $u \in H_2^\sigma \cap C^{\infty}_{b}(\bR^d)$ satisfy
\begin{equation*}
Lu - \lambda u = f \quad\text{in}\,\,\bR^d.
\end{equation*}
Then for all $\alpha \in (0, \min\{1,\sigma\})$,
\begin{multline}
                                \label{eq15.33}
\lambda\big(|u-(u)_{B_r}|\big)_{B_r} +
\big(|(-\Delta)^{\sigma/2}u-((-\Delta)^{\sigma/2}u)_{B_r}|\big)_{B_r}\\
\le N\kappa^{-\alpha}
\Big( \lambda \cM u (0)+\cM((-\Delta)^{\sigma/2}u)(0) \Big)+N\kappa^{d/2}\big(\cM(f^2)(0)\big)^{1/2},
\end{multline}
where $N=N(d,\nu,\Lambda,\sigma, \alpha)$.
\end{proposition}
\begin{proof}
Take a cut-off function $\eta\in C_0^\infty(B_{4\kappa r})$ such that $\eta=1$ in $B_{2\kappa r}$. Due to Proposition \ref{prop2.5}, there is a unique $H_2^\sigma$-solution to
$$
Lw-\lambda w=\eta f.
$$
Since $\eta f\in C_0^\infty$, by the classical theory, we know that $w\in H^\sigma_2 \cap C^{\infty}_{b}$. It follows from Lemma \ref{thm1.2} that
$$
 \lambda \|w\|_{L_2}+
\|(-\Delta)^{\sigma/2}w\|_{L_2} \le N(d,\nu) \|\eta f\|_{L_2},
$$
which yields, for any $R>0$,
\begin{multline}
                                    \label{eq15.46}
\Big(\lambda|w|+|(-\Delta)^{\sigma/2}w|\Big)_{B_{R}}
\le N(R^{-1}\kappa r)^{d/2}(f^2)_{B_{4\kappa r}}^{1/2}\\
\le N(R^{-1}\kappa r)^{d/2}\big(\cM(f^2)(0)\big)^{1/2}.
\end{multline}
Now $v:=u-w\in H_2^\sigma \cap C^{\infty}_{b}$ satisfies
$$
Lv-\lambda v=(1-\eta) f.
$$
Notice that $(1-\eta)f=0$ in $B_{2\kappa r}$.
By Corollary \ref{cor3.2}, we have
\begin{multline*}
\lambda\big(|v-(v)_{B_r}|\big)_{B_r}+ \big(|(-\Delta)^{\sigma/2}v-((-\Delta)^{\sigma/2}v)_{B_r}|\big)_{B_r}\\
\le  N\lambda\kappa^{-\alpha}\sum_{k=0}^{\infty} 2^{-k\sigma}
(|v|)_{B_{2^k \kappa r}}+
 N\kappa^{-\alpha}
\Big( \sum_{k=0}^{\infty} 2^{-k\sigma} (|(-\Delta)^{\sigma/2}v|)_{B_{2^k \kappa r}} + \cM f(0)\Big).
\end{multline*}
This together with the triangle inequality, \eqref{eq15.46}, and the inequality $\cM f(0)\le \big(\cM(f^2)(0)\big)^{1/2}$  gives
\begin{align*}
&\lambda\big(|u-(u)_{B_r}|\big)_{B_r} +
\big(|(-\Delta)^{\sigma/2}u-((-\Delta)^{\sigma/2}u)_{B_r}|\big)_{B_r}\\
&\le \lambda\big(|v-(v)_{B_r}|\big)_{B_r} +
\big(|(-\Delta)^{\sigma/2}v-((-\Delta)^{\sigma/2}v)_{B_r}|\big)_{B_r}\\
&\qquad +N\lambda\big(|w|\big)_{B_r} +
N\big(|(-\Delta)^{\sigma/2}w|\big)_{B_r}\\
&\le  N\kappa^{-\alpha}\sum_{k=0}^{\infty} 2^{-k\sigma}
\Big(\lambda|v|+|(-\Delta)^{\sigma/2}v|\Big)_{B_{2^k\kappa r}} + N\kappa^{d/2}
\big(\cM(f^2)(0)\big)^{1/2}\\
&\le  N\kappa^{-\alpha}\sum_{k=0}^{\infty} 2^{-k\sigma}
\Big(\lambda|u|+|(-\Delta)^{\sigma/2}u|\Big)_{B_{2^k \kappa r}}+ N\kappa^{d/2}
\big(\cM(f^2)(0)\big)^{1/2},
\end{align*}
which is  clearly less than the right-hand side of \eqref{eq15.33}. In the last inequality above, we used \eqref{eq15.46} with $R=2^k \kappa r,k=0,1,\ldots$. The proposition is proved.
\end{proof}

Next, we show that the inequality \eqref{eq15.33} holds true if we interchange the roles of $-(-\Delta)^{\sigma/2}$ and $L$.

\begin{lemma}
                            \label{lem3.4}
Let $\lambda > 0$, $0 <  \sigma < 2$, $r>0$, $\kappa\ge 2$,
and $f \in C^\infty_{\text{loc}}\cap L_{\infty}$.
Let $u \in H_2^\sigma \cap C^{\infty}_{b}(\bR^d)$ satisfy
\begin{equation}
                            \label{eq17.4.15}
-(-\Delta)^{\sigma/2}u - \lambda u = f \quad\text{in}\,\,\bR^d.
\end{equation}
Then for all $\alpha \in (0, \min\{1,\sigma\})$,
\begin{multline*}
              %                  \label{eq15.33zz}
\lambda\big(|u-(u)_{B_r}|\big)_{B_r} +
\big(|Lu-(Lu)_{B_r}|\big)_{B_r}\\
\le N\kappa^{-\alpha}
\Big( \lambda \cM u (0)+\cM(Lu)(0) \Big)+N\kappa^{d/2}\big(\cM(f^2)(0)\big)^{1/2},
\end{multline*}
where $N=N(d,\nu,\Lambda,\sigma, \alpha)$.
\end{lemma}

\begin{proof}
We follow the proof of Proposition \ref{prop3.3} with necessary changes outlined below. As before, we decompose $u$ as a sum of $w$ and $v$. For the estimate of $w$ corresponding to \eqref{eq15.46}, by using \eqref{eq0101z} and \eqref{eq24.08.55} we have
$$
\Big(\lambda|w|+|Lw|\Big)_{B_{R}}
\le N(R^{-1}\kappa r)^{d/2}\big(\cM(f^2)(0)\big)^{1/2}.
$$
Since the operator $L$ in Lemma \ref{lem3.1} can be set to be $(-\Delta)^{\sigma/2}$, one can still use \eqref{eq14.05} for the H\"older estimate of $v$.
Now for the H\"older estimate of $Lv$, we need an estimate similar to \eqref{eq14.08}:
$$
[Lu]_{C^{\alpha}(B_{1/2})} \le N
\left(\sum_{k=0}^{\infty} 2^{-k\sigma} (|Lu|)_{B_{2^k}} + \cM f(0)\right)
$$
provided that $f=0$ in $B_2$.
We apply $L$ to the both sides of \eqref{eq17.4.15} and obtain
$$
((-\Delta)^{\sigma/2}-\lambda)Lu = Lf.
$$
By Corollary \ref{cor2.3},
\begin{equation}
                                    \label{eq13.45z}
[Lu]_{C^{\alpha}(B_{1/2})}
\le N \|Lu\|_{L_1(\bR^d, \omega)}
+ N \sup_{B_1}|Lf|.
\end{equation}
We bound the first term on the right-hand side of \eqref{eq13.45z} as in the proof of Lemma \ref{lem3.1}. To estimate the second term, we notice that since $f=0$ in $B_2$, for any $|x|<1$ we have $\nabla f(x)=0$, and thus
\begin{align*}
\left|L f(x)\right|&=\left|\int_{\bR^d} \left(f(x+y) - f(x)-y\cdot \nabla f(x)\chi^{(\sigma)}(y)\right)K(y)\, d y\right|\\
&= \left|\int_{\bR^d} \left(f(x+y)-f(x)\right) K(y)\, dy\right|\\
                        \label{eq13.57z}
&\le N \int_{|y|>1/2} |f(x+y)|\frac{1}{|y|^{d+\sigma}+1} \, dy,
\end{align*}
which is bounded by $N\cM f(0)$ as desired. The remaining proof is the same as that of Proposition \ref{prop3.3}.
\end{proof}

\mysection{$L_p$-estimate} \label{L_p}

We finally complete the proof of the $L_p$ solvability of $Lu-\lambda u = f$ by providing the proof of Theorem \ref{mainth01}.

\begin{proof}[Proof of Theorem \ref{mainth01}]
First we prove the estimate \eqref{eq24.09.53} for $u \in C_0^{\infty}$ and $\lambda>0$.
In this case, clearly we have $u \in H_2^\sigma \cap C^{\infty}_{b}(\bR^d)$ and $f\in C^\infty_{\text{loc}}\cap L_{\infty}$.
When $p = 2$, the estimate is proved in Lemma \ref{thm1.2}.

Next we consider the case when $p \in (2,\infty)$.
Set $\alpha = \min\{1, \sigma\}/2$.
Then by Proposition \ref{prop3.3} combined with translations we have, for all $x \in \bR_d$,
$r > 0$ and $\kappa \ge 2$,
\begin{multline*}
\lambda\big(|u-(u)_{B_r(x)}|\big)_{B_r(x)} +
\big(|(-\Delta)^{\sigma/2}u-((-\Delta)^{\sigma/2}u)_{B_r(x)}|\big)_{B_r(x)}\\
\le N\kappa^{-\alpha}
\Big( \lambda \cM u (x)+\cM((-\Delta)^{\sigma/2}u)(x) \Big)+N\kappa^{d/2}\big(\cM(f^2)(x)\big)^{1/2},
\end{multline*}
where $N=N(d,\nu,\Lambda,\sigma)$. Take the supremum of the left-hand side of the inequality with respect to $r > 0$ to get
\begin{multline*}
\lambda u^{\#}(x) + \left((-\Delta)^{\sigma/2}u\right)^{\#}(x)\\
\le N\kappa^{-\alpha}
\Big( \lambda \cM u (x)+\cM((-\Delta)^{\sigma/2}u)(x) \Big)+N\kappa^{d/2}\big(\cM(f^2)(x)\big)^{1/2}.
\end{multline*}
By applying the Fefferman--Stein theorem on sharp functions and the Hardy--Littlewood maximal function theorem to the above inequality (see the inequalities in \eqref{eq1221}), we obtain
\begin{multline*}
\lambda \|u\|_{L_p}
+ \| (-\Delta)^{\sigma/2}u \|_{L_p}
\le N\lambda \|u^\#\|_{L_p}
+ \Big\| \big((-\Delta)^{\sigma/2}u\big)^\# \Big\|_{L_p}\\
\le N \kappa^{-\alpha} \left(\lambda \|\cM u\|_{L_p} + \Big\|\cM\big((-\Delta)^{\sigma/2}u\big)\Big\|_{L_p}\right)
+ N \kappa^{d/2} \|\cM(f^2)\|_{L_{p/2}}^{1/2}\\
\le N \kappa^{-\alpha} \left(\lambda \|u\|_{L_p} + \|(-\Delta)^{\sigma/2}u\|_{L_p}\right)
+ N \kappa^{d/2} \|f\|_{L_p},
\end{multline*}
where $N=N(d,\nu,\Lambda,\sigma,p)$.
It then only remains to take a sufficiently large $\kappa$ so that $N\kappa^{-\alpha} \le 1/2$.
For the case $\lambda = 0$ and $u\in C_0^\infty$, since the estimate \eqref{eq24.09.53} holds for any $\lambda > 0$, we take the limit as $\lambda \searrow 0$.
%{\color{red}
%\begin{equation}
%                                        \label{eq24.10.01}
%\|u\|_{\dot H_p^\sigma}
%\le N \|Lu\|_{L_p}.
%\end{equation}
%
%For any $\lambda_1>0$, by the proof above we have
%$$
%\|u\|_{\dot H_p^\sigma} + \sqrt{\lambda_1}\|u\|_{\dot{H}^{\sigma/2}_p}
%+ \lambda_1 \|u\|_{L_p}
%\le N \|Lu-\lambda_1\|_{L_p},
%$$
%where $N=N(d,\nu,\Lambda,\sigma_0,p)$. Taking the limit as $\lambda_1\to 0$ gives \eqref{eq24.10.01}.}

To prove \eqref{eq24.09.53} for general $u\in H_p^\sigma$, we need a continuity estimate of $L$ as in Lemma \ref{thm1.2}.
Thanks to Lemma \ref{lem3.4}, the argument using sharp and maximal functions as above yields, for any $\lambda>0$,
$$
\lambda\|u\|_{L_p}+\|Lu\|_{L_p}
\le N \|-(-\Delta)^{\sigma/2}u - \lambda u\|_{L_p},
$$
with a constant $N$ independent of $\lambda$.
Letting $\lambda\to 0$, we get for any $u\in C_0^\infty$,
\begin{equation}
                                        \label{eq24.10.01b}
\|Lu\|_{L_p}
\le N \|u\|_{\dot H_p^\sigma},
\end{equation}
which implies that $L$ is a continuous operator from $H_p^\sigma$ to $L_p$. Since $C_0^\infty$ is dense in $H_p^\sigma$, we obtain \eqref{eq24.09.53} in its full generality.

Now the unique solvability of the equation {in the case $p\in (2,\infty)$} follows from the same argument as in Proposition \ref{prop2.5} with $p$ in place of $2$ along with Lemma \ref{lem2.3} as well as the estimates \eqref{eq24.10.01b} and \eqref{eq24.09.53}.

For $p \in (1,2)$, we use a duality argument. Let $L^*$ be the non-local operator with kernel $K(-y)$. Denote $q=p/(p-1)\in (2,\infty)$. For any $g\in L_q$, by the $H^\sigma_q$-solvability there is a unique solution $v\in H^\sigma_q$ to the equation
$$
L^*v-\lambda v=g\quad \text{in}\,\,\bR^d.
$$
It is easily seen that $L^*$ is the adjoint operator of $L$. Therefore, for any $u\in C_0^\infty$,
\begin{align}
\int_{\bR^d}g(-\Delta)^{\sigma/2}u \,dx
&=\int_{\bR^d}(L^*v-\lambda v)(-\Delta)^{\sigma/2}u \,dx\nonumber\\
&=\int_{\bR^d}(-\Delta)^{\sigma/2}v (Lu-\lambda u)\,dx. \label{eq17.6.20}
\end{align}
By using \eqref{eq24.09.53} with $q$ in place of $p$, from \eqref{eq17.6.20} we have
\begin{align*}
\big|\int_{\bR^d}g(-\Delta)^{\sigma/2}u \,dx\big|
&\le \|(-\Delta)^{\sigma/2}v\|_{L_q}\|Lu-\lambda u\|_{L_p}\\
&\le N\|g\|_{L_q}\|Lu-\lambda u\|_{L_p}.
\end{align*}
Since $g\in L_q$ is arbitrary, we then get
$$
\|(-\Delta)^{\sigma/2}u\|_{L_p}\le N\|Lu-\lambda u\|_{L_p},
$$
which along with a similar estimate of $\lambda \|u\|_{L_p}$ yields \eqref{eq24.09.53} for any $u\in C_0^\infty$. For general $u\in H_p^\sigma$, as before we need a continuity estimate of $L$. For any $g\in L_q$, let $v\in H^\sigma_q$ be the equation
$$
-(-\Delta)^{\sigma/2}v-\lambda v=g\quad \text{in}\,\,\bR^d.
$$
For any $u\in C_0^\infty$, we have
\begin{align}
\int_{\bR^d}g Lu \,dx
&=\int_{\bR^d}\big(-(-\Delta)^{\sigma/2}v-\lambda v\big)Lu \,dx\nonumber\\
&=\int_{\bR^d}L^*v \big(-(-\Delta)^{\sigma/2}u-\lambda u\big)\,dx. \label{eq17.6.38}
\end{align}
By the continuity of $L^*$, from \eqref{eq17.6.38} we have
\begin{align*}
\big|\int_{\bR^d}gLu \,dx\big|
&\le \|L^*v\|_{L_q}\|-(-\Delta)^{\sigma/2}u-\lambda u\|_{L_p}\\
&\le N\|g\|_{L_q}\|(-\Delta)^{\sigma/2}u-\lambda u\|_{L_p}.
\end{align*}
Since $g\in L_q$ is arbitrary, we then get
$$
\|Lu\|_{L_p}\le N\|(-\Delta)^{\sigma/2}u-\lambda u\|_{L_p}.
$$
Letting $\lambda\to 0$ gives the continuity of $L$ from $H^\sigma_p$ to $L_p$. The rest of the proof is the same as in the case $p\in (2,\infty)$. The theorem is proved.
\end{proof}

\begin{proof}[Proof of the estimate \eqref{eq8.57}]
We take a smooth function $\eta\in C_0^\infty((-2,2))$ satisfying $\eta(t)=1$ for $t\in [-1,1]$. Fix a $T>0$. It is easily seen that $U(t,x):=\eta(t/T)u(x)\in C_0^\infty((-2T,2T)\times\bR^d)$ satisfies
$$
-D_tU(t,x)+LU(t,x)+b\cdot \nabla U(t,x)-\lambda U(t,x)=\eta(t/T)f(x)-u(x)\eta'(t/T)/T.
$$
Define $V(t,x)=U(t,x-bt)$. Then $V\in C_0^\infty((-2T,2T)\times\bR^d)$ and satisfies
$$
-D_tV(t,x)+LV(t,x)-\lambda V(t,x)=\eta(t/T)f(x-bt)-u(x-bt)\eta'(t/T)/T.
$$
It follows from the results in \cite{MP10} combined with the continuity $L$ from $H_p^\sigma$ to $L_p$ proved in Theorem \ref{mainth01} (see Remark \ref{rem11.19}) that
\begin{multline*}
\|(-\Delta)^{\sigma/2}V\|_{L_p(((-2T,2T)\times\bR^d)}+\lambda\|V\|_{L_p(-2T,2T)\times\bR^d)}\\
\le N\|\eta(t/T)f(x-bt)-u(x-bt)\eta'(t/T)/T\|_{L_p((-2T,2T)\times\bR^d)},
\end{multline*}
which implies
$$
\|(-\Delta)^{\sigma/2}u\|_{L_p}+\lambda\|u\|_{L_p}\\
\le N\|f\|_{L_p}+NT^{-1}\lambda\|u\|_{L_p}
$$
with a constant $N=N(d,\nu,\Lambda,\sigma,p)$. Letting $T\to \infty$, we get
$$
\|u\|_{\dot H_p^\sigma} + \sqrt{\lambda}\|u\|_{\dot{H}^{\sigma/2}_p}
+ \lambda \|u\|_{L_p}\le N \|f\|_{L_p}.
$$
To complete the proof, we use the equation \eqref{eq8.56} and \eqref{eq24.10.01b} to bound the $L_p$ norm of $b\cdot \nabla u$ by
$$
\|Lu\|_{L_p}+\lambda\|u\|_{L_p}+ \|f\|_{L_p}\le N \|f\|_{L_p}.
$$
\end{proof}

%\begin{remark}
%                                            \label{rem5.2}
%As is seen in the proof of Theorem \ref{mainth01}, we only need to verify two things: a continuity estimate of $L$ from $H^\sigma_2$ to $L_2$ and a H\"older estimate of $Lu$.  Next, for the H\"older estimate of $Lu$, again we interchange the roles of $-(-\Delta)^{\sigma/2}$ and $L$ in the arguments in Section \ref{moe}. Compared to the case of symmetric kernels,
%The remaining proof is the same.
%\end{remark}

\mysection{Local estimates}
                                                \label{sec6}
From the global estimate in Theorem \ref{mainth01}, by using a more or less standard localization argument one can obtain the following
interior estimates.
\begin{equation}
                                \label{eq21.06}
\|(-\Delta)^{\sigma/2}u\|_{L_p(B_1)}\le
N\|f\|_{L_p(B_2)}+N\|u\|_{L_p(\bR^d,\omega)}
\end{equation}
for $\sigma\in (0,1)$,
\begin{equation}
                                \label{eq21.06b}
\|(-\Delta)^{\sigma/2}u\|_{L_p(B_1)}\le
N\|f\|_{L_p(B_2)}+N\|u\|_{L_p(\bR^d,\omega)} +N\|Du\|_{L_p(B_4)}
\end{equation}
for $\sigma\in (1,2)$, and
\begin{equation}
                                \label{eq21.06c}
\|(-\Delta)^{\sigma/2}u\|_{L_p(B_1)}\le N\|f\|_{L_p(B_2)}+
N(\varepsilon)\|u\|_{L_p(\bR^d,\omega)}+\varepsilon\|Du\|_{L_p(B_4)}
\end{equation}
for $\sigma=1$ and any $\varepsilon\in (0,1)$. Here the weight
function $\omega$ is defined in Theorem \ref{thm1201}.

For the proof of this claim, we take a cut-off function $\eta\in
C_0^\infty(B_2)$ satisfying $\eta\equiv 1$ on $B_1$. Then it is
easily seen that
$$
L(\eta u)-\lambda \eta u=\eta f+L(\eta u)-\eta Lu.
$$
Applying the global estimate in Theorem \ref{mainth01} to the equation above gives
\begin{multline*}
\|(-\Delta)^{\sigma/2}(\eta u)\|_{L_p(\bR^d)}\le N\|\eta f+L(\eta u)-\eta Lu\|_{L_p(\bR^d)}\\
\le N\|f\|_{L_p(B_2)}+N\|L(\eta u)-\eta Lu\|_{L_p(\bR^d)}.
\end{multline*}
Thus, by the triangle inequality,
\begin{multline}
                        \label{eq21.39}
\|(-\Delta)^{\sigma/2}u\|_{L_p(B_1)}\le \|\eta(-\Delta)^{\sigma/2}u\|_{L_p(\bR^d)}
\le N\|f\|_{L_p(B_2)}\\
+N\|L(\eta u)-\eta Lu\|_{L_p(\bR^d)}
+\|(-\Delta)^{\sigma/2}(\eta u)-\eta (-\Delta)^{\sigma/2}u\|_{L_p(\bR^d)}.
\end{multline}
It suffices to estimate the second term on the right-hand side above since the estimate of the third term is similar. We compute
\begin{multline*}
L(\eta u)-\eta Lu\\
=\int_{\bR^d}\Big(\big(\eta(x+y)-\eta(x)\big)u(x+y)-y\cdot \nabla \eta(x) u(x)\chi^{(\sigma)}(y)\Big)K(y)\,dy.
\end{multline*}

{\em (i)} For $\sigma\in (0,1)$, we have
\begin{align*}
|L(\eta u)-\eta Lu|&\le \int_{\bR^d}\big|(\eta(x+y)-\eta(x))u(x+y)\big|K(y)\,dy\\
&\le N\left(\int_{B_1}+\int_{B_1^c}\right)
\big|(\eta(x+y)-\eta(x))|u(x+y)|\big||y|^{-d-\sigma}\,dy.
\end{align*}
By using the obvious bound
\begin{equation}
                                                \label{eq17.01}
|\eta(x+y)-\eta(x)|\le N|y|1_{|x|<3}\quad \text{for}\,\, y \in B_1, %\quad
%|\eta(x+y)-\eta(x)|\le N\,\, \text{in}\, B_1^c,
\end{equation}
we get
\begin{multline}
							\label{eq20_02}
|L(\eta u)-\eta Lu|\le
N\int_{B_1}1_{|x|<3}|u(x+y)||y|^{1-d-\sigma}\,dy\\
+\int_{B_1^c}|u(x+y)|\big(1_{|x+y|<2}+1_{|x|<2}\big)|y|^{-d-\sigma}\,dy.
\end{multline}
By Minkowski's inequality and H\"older's inequality,
\begin{equation}
							\label{eq20_01}
\|L(\eta u)-\eta Lu\|_{L_p(\bR^d)}\le N\|u\|_{L_p(\bR^d,\omega)},
\end{equation}
which together with \eqref{eq21.39} yields \eqref{eq21.06}.
Indeed, to obtain the above estimate the last term in \eqref{eq20_02} is calculated as follows.
\begin{multline*}
\bigg\| \int_{|y|>1} 1_{|x|<2} |u(x+y)| |y|^{-d-\sigma} \, dy \bigg\|_{L_p(\bR^d)}
\le 2 \int_{\bR^d} \|u(\cdot+y)\|_{L_p(B_2)} \omega(y) \, dy\\
\le 2 \left( \int_{\bR^d} \|u\|_{L_p(B_2(y))}^p \omega(y) \, dy \right)^{1/p} \left( \int_{\bR^d} \omega(y) \, dy \right)^{1/q}\\
\le N \left( \int_{\bR^d} |u(x)|^p \int_{B_2(x)} \omega(y) \, dy \, dx\right)^{1/p} \le N \|u\|_{L_p(\bR^d,\omega)},
\end{multline*}
where $q=p/(p-1)$.

{\em (ii)} For $\sigma\in (1,2)$, we have
\begin{align}
|L(\eta u)-\eta Lu|&\le \int_{\bR^d}\big|(\eta(x+y)-\eta(x))u(x+y)-y\cdot \nabla \eta(x)u(x)\big|K(y)\,dy\nonumber\\
&\le I_1+I_2,                       \label{eq22.05}
\end{align}
where
\begin{align*}
I_1&:=\int_{\bR^d}\big|(\eta(x+y)-\eta(x))(u(x+y)-u(x))\big|K(y)\,dy,\\
I_2&:=\int_{\bR^d}\big|(\eta(x+y)-\eta(x)-y\cdot \nabla \eta(x))u(x)\big|K(y)\,dy.
\end{align*}
Note that
$$
|u(x+y)-u(x)|\le |y|\int_0^1 |\nabla u(x+ty)|\,dt.
$$
We use \eqref{eq17.01} and the bound above to estimate $I_1$ by
\begin{align*}
I_1:&=\Big(\int_{B_1}+\int_{B_1^c}\Big)
\big|(\eta(x+y)-\eta(x))(u(x+y)-u(x))\big|K(y)\,dy\\
&\le N\int_{B_1}\int_0^1 1_{|x|<3}|\nabla
u(x+ty)||y|^{2-d-\sigma}\,dt\,dy\\
&\quad +N\int_{B_1^c}\big(|u(x+y)|+|u(x)|\big)\big(1_{|x+y|<2}+1_{|x|<2}\big)|y|^{-d-\sigma}\,dy.
\end{align*}
By Minkowski's inequality and H\"older's inequality as used for \eqref{eq20_01},
\begin{equation}
                                \label{eq22.06}
\|I_1\|_{L_p(\bR^d)}\le
N\|Du\|_{L_p(B_4)}+N\|u\|_{L_p(\bR^d,\omega)}.
\end{equation}
Note that by the mean value theorem,
\begin{equation*}
\big|\eta(x+y)-\eta(x)-y\cdot \nabla \eta(x)\big|\le
N|y|^21_{|x|<3}\quad\text{for}\,\,y \in B_1.
\end{equation*}
Thus we have
\begin{multline*}
I_2\le N|u(x)|1_{|x|<3}\int_{B_1}|y|^{2-d-\sigma}\,dy\\
+N|u(x)|\int_{B_1^c}\big(1_{|x+y|<2}+1_{|x|<2}(1+|y|)\big)|y|^{-d-\sigma}\,dy.
\end{multline*}
Again, by Minkowski's inequality and H\"older's inequality,
$$
\|I_2\|_{L_p(\bR^d)}\le N\|u\|_{L_p(\bR^d,\omega)},
$$
which together with \eqref{eq22.05} and \eqref{eq22.06} gives
$$
\|L(\eta u)-\eta Lu\|_{L_p(\bR^d)}\le
N\|u\|_{L_p(\bR^d,\omega)}+N\|Du\|_{L_p(B_4)},
$$
and thus \eqref{eq21.06b}.

{\em (iii)} In the last case $\sigma=1$, by using \eqref{eq21.47},
for any $\delta\in (0,1)$ we have
$$
|L(\eta u)-\eta Lu|\le I_3+I_4+I_5,
$$
where
\begin{align*}
I_3&:=\int_{B_\delta}\big|(\eta(x+y)-\eta(x))(u(x+y)-u(x))\big|K(y)\,dy,\\
I_4&:=\int_{B_\delta}\big|(\eta(x+y)-\eta(x)-y\cdot \nabla \eta(x))u(x)\big|K(y)\,dy,\\
I_5&:=\int_{B_\delta^c}\big|(\eta(x+y)-\eta(x))u(x+y)\big|K(y)\,dy.
\end{align*}
We bound $I_3$ and $I_4$ in the same way as $I_1$ and $I_2$ to get
\begin{align*}
I_3&\le N\int_{B_\delta}\int_0^1 1_{|x|<3}|\nabla u(x+ty)||y|^{1-d}\,dt\,dy,\\
I_4&\le N\int_{B_\delta}1_{|x|<3}|u(x)||y|^{1-d}\,dy,
\end{align*}
and bound $I_5$ as in the first case to get
\begin{multline*}
I_5\le N\int_{B_1\setminus
B_\delta}1_{|x|<3}|u(x+y)||y|^{-d}\,dy\\
+N\int_{B_1^c}|u(x+y)|(1_{|x+y|<2}+1_{|x|<2})|y|^{-d-1}\,dy.
\end{multline*}
Thus, by Minkowski's inequality and H\"older's inequality,
$$
\|L(\eta u)-\eta Lu\|_{L_p(\bR^d)}\le N\delta\|Du\|_{L_p(B_4)}
+N(1-\log(\delta))\|u\|_{L_p(\bR^d,\omega)}.
$$
By choosing a suitable $\delta$, we obtain \eqref{eq21.06c}. The claim is proved.

%===============================================================================%

\section*{Acknowledgement}
The authors would like to thank the referee for the careful reading of the manuscript and many useful comments.

\end{document}